\documentclass[11pt]{article}
\usepackage[tbtags]{amsmath}
\usepackage{amssymb}
\usepackage{amsthm}
\usepackage[misc]{ifsym}
\usepackage{graphicx}
\usepackage{tikz}
\usetikzlibrary{shapes,arrows}
\usepackage{cases}
\usepackage{color}
\numberwithin{equation}{section}

\newtheorem{theorem}{Theorem}[section]
\newtheorem{lemma}{Lemma}[section]

\newtheorem{remark}{Remark}[section]
\normalsize
\usepackage{hyperref,amsmath,amssymb,amscd}
\title{\bf Stackelberg Stochastic Differential Game with Asymmetric Noisy Observations
\thanks{This work is supported by National Key R\&D Program of China (Grant No. 2018YFB1305400) and National Natural Science Foundations of China (Grant No. 11971266, 11831010, 11571205).}}
\author{\normalsize Yueyang Zheng\thanks{\it School of Mathematics, Shandong University, Jinan 250100, P.R.China, E-mail: zhengyueyang0106@163.com} , Jingtao Shi\thanks{Corresponding author. \it School of Mathematics, Shandong University, Jinan 250100, P.R.China, E-mail: shijingtao@sdu.edu.cn}}

\begin{document}
\maketitle

\noindent{\bf Abstract:}\quad This paper is concerned with a Stackelberg stochastic differential game with asymmetric noisy observation, with one follower and one leader. In our model, the follower cannot observe the state process directly, but could observe a noisy observation process, while the leader can completely observe the state process. Open-loop Stackelberg equilibrium is considered. The follower first solve an stochastic optimal control problem with partial observation, the maximum principle and verification theorem are obtained. Then the leader turns to solve an optimal control problem for a conditional mean-field forward-backward stochastic differential equation, and both maximum principle and verification theorem are proved. An linear-quadratic Stackelberg stochastic differential game with asymmetric noisy observation is discussed to illustrate the theoretical results in this paper. With the aid of some Riccati equations, the open-loop Stackelberg equilibrium admits its state estimate feedback representation.

\vspace{1mm}


\vspace{1mm}


\section{Introduction}

The Stackelberg game is an important type of hierarchical noncooperative games (Ba\c{s}ar and Olsder \cite{BO98}), whose study can be traced back to the pioneering work by Stackelberg \cite{S52}. The Stackelberg game is usually know as the leader-follower game, whose economic background comes from markets where some firms have power of domination over others. The solutions of the Stackelberg differential game, are called Stackelberg equilibrium points in which there are usually two players with asymmetric roles, one leader and one follower. In order to obtain the Stackelberg equilibrium points, it is usual to divide the game problem into two parts. In the first part---the follower's problem, firstly the leader announces his strategy, then the follower will make an instantaneous response, and choose an optimal strategy corresponding to the given leader's strategy to optimize his/her cost functional. In the second part---the leader's problem, knowing the follower would take such an optimal strategy, the leader will choose an optimal strategy to optimize his/her cost functional. In a word, a distinctive feature of the Stackelberg differential games is that, the decisions must be made by two players and one of them is subordinated to the other because of the asymmetric roles, therefore one player must make a decision after the other player's decision is made. The Stackelberg game has been widely applied in the principal-agent/optimal contract problems (Cvitani\'{c} and Zhang \cite{CZ13}), the newsvendor/wholesaler problems (\O ksendal et al. \cite{OSU13}) and optimal reinsurance problems (Chen and Shen \cite{CS18}).

There exist some literatures about the Stackelberg differential game for It\^{o}'s {\it stochastic differential equations} (SDEs for short) in the past decades. Let us mention a few. Yong \cite{Yong02} studied the indefinite {\it linear-quadratic} (LQ for short) leader-follower differential game with random coefficients and control-dependent diffusion. {\it Forward-backward stochastic differential equations} (FBSDEs for short) and Riccati equations are applied to obtain the state feedback representation of the open-loop Stacklelberg equilibrium points. Bensoussan et al. \cite{BCS15} introduced several solution concepts in terms of the players' information sets, and studied LQ Stackelberg games under both adapted open-loop and closed-loop memoryless information structures, whereas the control variables do not enter the diffusion coefficient of the state equation. Mukaidani and Xu \cite{MX15} studied the Stackelberg game with one leader and multiple followers, in an infinite time horizon. The Stackelberg equilibrium points are developed, by cross-coupled algebraic Riccati equations, under both cooperative and non-cooperative settings of the followers, to attain Pareto optimality and Nash equilibrium, respectively. Xu and Zhang \cite{XZ16} and Xu et al. \cite{XSZ18} investigated the LQ Stackelberg differential games with time delay. Li and Yu \cite{LY18} proved the solvability of a kind of coupled FBSDEs with a multilevel self-similar domination-monotonicity structure, then it is used to characterize the unique equilibrium of an LQ generalized Stackelberg stochastic differential game with hierarchy in a closed form. Moon and Ba\c{s}ar \cite{MB18}, Lin et al. \cite{LJZ19} studied the the LQ mean-field Stackelberg stochastic differential games. Du and Wu \cite{DW19} investigated an LQ Stackelberg game of mean-field {\it backward stochastic differential equations} (BSDEs for short). Zheng and Shi \cite{ZS19} researched the Stackelberg game of BSDEs with complete information. Feng et al. \cite{FHH20} considered the LQ Stackelberg game of BSDEs with constraints.

However, in all the above literatures about the Stackelberg game, the authors assume that both the leader and the follower could fully observe the state of the controlled stochastic systems. Obviously, this is not practical in reality. Generally speaking, the players in the games can only obtain partial information in most cases. Then it is very natural to study the Stackelberg stochastic differential game under partial information. In fact, some efforts have been made such as the following. Shi et al. \cite{SWX16,SWX17} studied the Stackelberg stochastic differential game and introduced a new explanation for the asymmetric information feature, that the information available to the follower is based on the some sub-$\sigma$-algebra of that available to the leader. Shi et al. \cite{SWX20} investigated the LQ Stackelberg stochastic differential game with overlapping information, where the follower's and the leader's information have some joint part, while they have no inclusion relations. Wang et al. \cite{WWZ20} discussed an asymmetric information mean-field type LQ Stackelberg stochastic differential game with one leader and two followers.

Noting that in the game frameworks of papers \cite{WY12,CX14,SWX16,SWX17,WXX18,SWX20,WWZ20}, the information available to the players are described by the filtration generated by standard Brownian motions. In fact, in realty there exists many situations, where only some observation processes could be observed by the players. For example, in the financial market, the investors can only observe the security prices. Thus the portfolio process is required to be adapted to the natural filtration of the security price process (Xiong and Zhou \cite{XZ07}). In general, partially observed problems are related with filtering theory (Liptser and Shiryayev \cite{LS77}, Bensoussan \cite{Ben92}, Xiong \cite{Xiong08}). Partially observed stochastic optimal control and differential games have been researched by many authors, such as Li and Tang \cite{LT95}, Tang \cite{Tang98}, Wang and Wu \cite{WW09}, Huang et al. \cite{HWX09}, Wu \cite{Wu10}, Shi and Wu \cite{SW10}, Wang et al. \cite{WWX13,WWX15,WWX18}, Wu and Zhuang \cite{WZ18}.

Inspired by the above literatures, in this paper we study the Stackelberg differential game with asymmetric noisy observation, with deterministic coefficients and convex control domains. To the best of our knowledge, papers on the topic about partially observed Stackelberg differential games are quite lacking, except Li et al. \cite{LFCM19}. Note that in \cite{LFCM19}, the leader-follower Stackelberg stochastic differential game under a symmetric, partial observed information is researched. The novelty of the formulation and the contribution in this paper is the following.

(1) A new kind of Stackelberg stochastic differential game with asymmetric noisy observation is introduced. In our framework, the control processes of the follower are required to be adapted to the information filtration generated by the observation process, which is a Brownian motion is the original probability space, while the information filtration available to the leader is generated by both the Brownian noise and the observation process.

(2) For the follower's problem, a stochastic optimal control problem with partial observation is solved. The partial information maximum principle (Theorem 3.1) is given, which is direct from Bensoussan \cite{Ben92}, Li and Tang \cite{LT95}. Thanks for a mild assumption motivated by Huang et al. \cite{HLW10}, the partial information verification theorem (Theorem 3.2) is proved. It is remarkable that the Hamiltonian function \eqref{H1} and adjoint equations \eqref{adjoint1}, \eqref{adjoint11} are different from those in \cite{LFCM19}, but similar as \cite{LT95}.

(3) For the leader's problem, a stochastic optimal control problem of FBSDE is solved. Since the control processes are required by the information filtration generated by both the Brownian motion and the observation process, we encounter a difficulty when applying the techniques in Wu \cite{Wu10} and Wang et al. \cite{WWX13}. We overcome this difficulty by again the mild assumption used in Theorem 3.2 and Bayes' formula, to obtain the maximum principle of the leader (Theorem 3.3). However, by Clarke's generalized gradient, we could prove the verification theorem (Theorem 3.4) of the leader only in the special case, since the difficulty is fatal.

(4) For the LQ case, it consists of an LQ stochastic optimal control problem with partial observation for the follower, and followed by an LQ stochastic optimal control problem of the coupled conditional mean-field FBSDE with complete observation information for the leader. The state estimate feedback representation of the Stackelberg equilibrium is obtained, via some Riccati equations, by Theorems 3.1-3.4, and the technique of Yong \cite{Yong02}.

The rest of this paper is organized as follows. In Section 2, the Stackelberg stochastic differential game with asymmetric noisy observation is formulated. In Section 3, maximum principles and verification theorems are proved, for the problems of the follower and the leader, respectively. Then the LQ Stackelberg stochastic differential game with asymmetric noisy observation is investigated in Section 4. Specially, Subsection 4.1 is devoted to the solution to an LQ stochastic optimal control problem with partial observation of the follower. Subsection 4.2 is devoted to the solution to an LQ stochastic optimal control problem of coupled conditional mean-field FBSDE with complete observation information of the leader. The open-loop Stackelberg equilibrium is represented as its state estimate feedback form. Finally, Section 5 gives some concluding remarks.

\section{Problem formulation}

Let $T>0$ be be a finite time duration. Let $(\Omega,\mathcal{F},\mathbb{P})$ be a probability space on which two independent standard Brownian motions $W(\cdot)$ and $Y(\cdot)$ valued in $\mathbb{R}^{d_1}$ and $\mathbb{R}^{d_2}$ are defined. For $t\geq0$, $\mathcal{F}_t^{W}$ and $\mathcal{F}_t^{Y}$ are the natural filtration generated by $W(\cdot)$ and $Y(\cdot)$, respectively, and we set $\mathcal{F}_t=\mathcal{F}_{t}^{W}\times\mathcal{F}_{t}^{Y}$. $\mathbb{E}$ denotes the expectation under probability $\mathbb{P}$. In this paper, $L_{\mathcal{F}_T}^2(\Omega,\mathbb{R}^n)$ denotes the set of $\mathbb{R}^n$-valued, $\mathcal{F}_T$-measurable, square-integrable random variables, $L^2_\mathcal{F}(0,T;\mathbb{R}^n)$ denotes the set of $\mathbb{R}^n$-valued, $\mathcal{F}_t$-adapted, square integrable processes on $[0,T]$, and $L^\infty(0,T;\mathbb{R}^n)$ denotes the set of $\mathbb{R}^n$-valued, bounded functions on $[0,T]$.

\vspace{1mm}

Let us consider the following controlled {\it stochastic differential equation} (SDE, for short):
\begin{equation}\label{sde}
\left\{
\begin{aligned}
dx^{u_1,u_2}(t)&=b(t,x^{u_1,u_2}(t),u_1(t),u_2(t))dt+\sigma(t,x^{u_1,u_2}(t),u_1(t),u_2(t))dW(t),\ t\in[0,T],\\
  x^{u_1,u_2}(0)&=x_0,
\end{aligned}
\right.
\end{equation}
where $u_1(\cdot)$ and $u_2(\cdot)$ are control processes taken by the two players in the game, labeled 1 (the follower) and 2 (the leader) with values in nonempty convex sets $U_1\subseteq\mathbb{R}^{m_1}$ and $U_2\subseteq\mathbb{R}^{m_2}$, respectively. Here, $x_0\in\mathbb{R}^n$, $b:[0,T]\times\mathbb{R}^{n}\times U_1\times U_2\rightarrow\mathbb{R}^n$, $\sigma:[0,T]\times\mathbb{R}^{n}\times U_1\times U_2\rightarrow\mathbb{R}^{n\times d_1}$ are given functions.

We assume that the state process $x^{u_1,u_2}(\cdot)$ cannot be observed by the follower directly, but he/she can observe a related process $Y(\cdot)$, which satisfies the following controlled stochastic system:
\begin{equation}\label{observation eq}
Y(t)=\int_0^th(s,x^{u_1,u_2}(s),u_1(s),u_2(s))ds+W^{u_1,u_2}(t),
\end{equation}
where $h(t,x,u_1,u_2):[0,T]\times\mathbb{R}^{n}\times U_1\times U_2\rightarrow\mathbb{R}^{d_2}$ are give functions, and $W^{u_1,u_2}(\cdot)$ denotes a stochastic process depending on the control process pair $(u_1(\cdot),u_2(\cdot))$.

\vspace{1mm}

The following hypotheses are assumed.

{\bf(A1)} The functions $b,\sigma$ are linear growth and continuously differentiable with respect to $u_1,u_2$ and $x$, and their partial derivatives with respect to $u_1,u_2$ and $x$ are all uniformly bounded. Moreover, the function $h$ is continuously differentiable with respect to $u_1,u_2$ and $x$, and there exists some constant $K>0$, such that for any $t\in[0,T],x\in\mathbb{R}^n,u_1\in\mathbb{R}^{d_1},u_2\in\mathbb{R}^{d_2}$,
\begin{equation*}
|h(t,x,u_1,u_2)|+|h_x(t,x,u_1,u_2)|+|h_{u_1}(t,x,u_1,u_2)|+|h_{u_2}(t,x,u_1,u_2)|\leq K.
\end{equation*}

Motivated by some interesting random phenomena in realty, we begin to explain the asymmetric information between the follower and the leader, in our Stackelberg game problem. In the follower's problem, a stochastic optimal control problem with partial information need to be solved, since the information available to him/her at time $t$ is based on the filtration generated by the noisy observation process $\mathcal{F}_t^Y=\sigma\{Y(s),0\leq s\leq t\}$. However, in the leader's problem, a stochastic optimal control problem with complete information is required to be solved, since the information available to him/her at time $t$ is based on the complete information/filtration $\mathcal{F}_t$. Obviously, we have $\mathcal{F}_t^{Y}\subseteq\mathcal{F}_t$ and the information of the follower and the leader has the asymmetric feature and structure.

Next, we define the admissible control sets of the follower and the leader, respectively, as follows:
\begin{equation}\label{admissible controls}
\begin{aligned}
\rm\mathcal{U}_1&=\Big\{u_1\Big|u_1: \Omega\times[0,T]\rightarrow U_1 \mbox{ is }\mathcal{F}_t^Y\mbox{-adapted and }
  \sup_{0\leq t\leq T}\mathbb{E}|u_1(t)|^i<\infty,i=1,2,\cdots\Big\},\\
\rm\mathcal{U}_2&=\Big\{u_2\Big|u_2: \Omega\times[0,T]\rightarrow U_2 \mbox{ is }\mathcal{F}_t\mbox{-adapted and }\sup_{0\leq t\leq T}\mathbb{E}|u_2(t)|^i<\infty,i=1,2,\cdots\Big\}.
\end{aligned}
\end{equation}
For any $(u_1(\cdot),u_2(\cdot))\in\mathcal{U}_1\times\mathcal{U}_2$, we know that \eqref{sde} admits a unique solution under hypothesis {\bf(A1)}, which is denoted by $x^{u_1,u_2}(\cdot)\in L^2_\mathcal{F}(0,T;\mathbb{R}^n)$.

From Girsanov's theorem, it follows that if we define
\begin{equation}\label{Z-u1u2}
\begin{aligned}
Z^{u_1,u_2}(t)&:=\exp\bigg\{\int_0^th^{\top}(s,x^{u_1,u_2}(s),u_1(s),u_2(s))dY(s)\\
              &\qquad\qquad-\frac{1}{2}\int_0^T\big|h(s,x^{u_1,u_2}(s),u_1(s),u_2(s))\big|^2ds\bigg\},
\end{aligned}\end{equation}
i.e.,
\begin{equation}\label{dZ}
\left\{
\begin{aligned}
dZ^{u_1,u_2}(t)&=Z^{u_1,u_2}(t)h^{\top}(t,x^{u_1,u_2}(t),u_1(t),u_2(t))dY(t),\ t\in[0,T],\\
  Z^{u_1,u_2}(0)&=1,
\end{aligned}
\right.
\end{equation}
and if $d\mathbb{P}^{u_1,u_2}:=Z^{u_1,u_2}(T)d\mathbb{P}$, then $\mathbb{P}^{u_1,u_2}$ is a new probability and $(W(\cdot),W^{u_1,u_2}(\cdot))$ is an $\mathbb{R}^{d_1+d_2}$-valued Brownian motion under $\mathbb{P}^{u_1,u_2}$.

In our Stackelberg game problem, knowing that the leader has chosen $u_2(\cdot)\in\mathcal{U}_2$, the follower would like to choose an $\mathcal{F}_t^Y$-adapted control $\bar{u}_1(\cdot)\equiv\bar{u}_1(\cdot;u_2(\cdot))$ to minimize his cost functional
\begin{equation}\label{cf1}
J_1(u_1(\cdot),u_2(\cdot))=\mathbb{E}^{u_1,u_2}\bigg[\int_0^Tl_1(t,x^{u_1,u_2}(t),u_1(t),u_2(t))dt+G_1(x^{u_1,u_2}(T))\bigg],
\end{equation}
where $\mathbb{E}^{u_1,u_2}$ denotes the expectation under the probability $\mathbb{P}^{u_1,u_2}$. Here functions $l_1:[0,T]\times\mathbb{R}^n\times U_1\times U_2\rightarrow\mathbb{R}$ and $G_1:\mathbb{R}^n\rightarrow\mathbb{R}$ are given.

{\bf Problem of the follower}. For any chosen $u_2(\cdot)\in\mathcal{U}_2$ by the leader, choose an $\mathcal{F}_t^Y$-adapted control $\bar{u}_1(\cdot)=\bar{u}_1(\cdot;u_2(\cdot))\in\mathcal{U}_1$ such that
\begin{equation}\label{follower}
J_1(\bar{u}_1(\cdot),u_2(\cdot))\equiv J_1(\bar{u}_1(\cdot;u_2(\cdot)),u_2(\cdot))=\inf_{u_1\in\mathcal{U}_1}J_1(u_1(\cdot),u_2(\cdot)),
\end{equation}
subject to \eqref{sde} and \eqref{cf1}. Such a $\bar{u}_1(\cdot)=\bar{u}_1(\cdot;u_2(\cdot))$ is called an optimal control, and the corresponding solution $x^{\bar{u}_1,u_2}(\cdot)$ to \eqref{sde} is called an optimal state process, for the follower.

In the following procedure of the game problem, once knowing that the follower would take such an optimal control $\bar{u}_1(\cdot)=\bar{u}_1(\cdot;u_2(\cdot))$, the leader would like to choose an $\mathcal{F}_t$-adapted control $\bar{u}_2(\cdot)$ to minimize his cost functional
\begin{equation}\label{cf2}
J_2(\bar{u}_1(\cdot),u_2(\cdot))=\mathbb{E}^{\bar{u}_1,u_2}\bigg[\int_0^Tl_2(t,x^{\bar{u}_1,u_2}(t),\bar{u}_1(t),u_2(t))dt+G_2(x^{\bar{u}_1,u_2}(T))\bigg].
\end{equation}
Here functions $l_2:[0,T]\times\mathbb{R}^n\times U_1\times U_2\rightarrow\mathbb{R}$ and $G_2:\mathbb{R}^n\rightarrow\mathbb{R}$ are given.

{\bf Problem of the leader}. Find an $\mathcal{F}_t$-adapted control $\bar{u}_2(\cdot)\in\mathcal{U}_2$ such that
\begin{equation}\label{leader}
J_2(\bar{u}_1(\cdot),\bar{u}_2(\cdot))\equiv J_2(\bar{u}_1(\cdot;\bar{u}_2(\cdot)),\bar{u}_2(\cdot))=\inf_{u_2\in\mathcal{U}_2}J_2(\bar{u}_1(\cdot;u_2(\cdot)),u_2(\cdot)),
\end{equation}
subject to \eqref{sde} and \eqref{cf2}. Such a $\bar{u}_2(\cdot)$ is called an optimal control, and the corresponding solution $x^{\bar{u}_1,\bar{u}_2}(\cdot)$ to \eqref{sde} is called an optimal state process, for the leader. We will restate the problem for the leader in more detail, since its precise description has to involve the solution to {\bf Problem of the follower}.

We refer to the problem mentioned above as a {\it Stackelberg stochastic differential game with asymmetric noisy observations}. If there exists a control process pair $(\bar{u}_1(\cdot),\bar{u}_2(\cdot))$ satisfy \eqref{follower} and \eqref{leader}, we refer to it as an {\it open-loop Stackelberg equilibrium}.

We also introduce the following assumption.

{\bf(A2)}\ For $i=1,2$, the functions $l_i,G_i$ are continuously differentiable with respect to $x,u_1,u_2$, and there exists a constant $C>0$ such that for any $t\in[0,T],x\in\mathbb{R}^n,u_1\in\mathbb{R}^{d_1},u_2\in\mathbb{R}^{d_2}$,
\begin{equation*}\begin{aligned}
&\big(1+|x|^2+|u_1|^2+|u_2|^2\big)^{-1}|l_i(t,x,u_1,u_2)|+\big(1+|x|+|u_1|+|u_2|\big)^{-1}\big(|l_{ix}(t,x,u_1,u_2)|\\
&\quad+|l_{iu_1}(t,x,u_1,u_2)|+|l_{iu_2}(t,x,u_1,u_2)|\big)\leq C,\\
&(1+|x|^2)^{-1}|G_i(x)|+(1+|x|)^{-1}|G_{ix}(x)|\leq C.
\end{aligned}\end{equation*}

\section{Maximum principle and verification theorem for Stackelberg equilibrium}

In this paper, we frequently omit some time variable $t$ in some mathematical formula for simplicity, if there exists no ambiguity.

\subsection{The problem of the follower}

For any chosen $u_2(\cdot)\in\mathcal{U}_2$, we first consider {\bf Problem of the follower} which is a partially observed stochastic optimal control problem.

By Bayes's formula, {\bf Problem of the follower} is equivalent to minimize
\begin{equation}\label{cf11}
J_1(u_1(\cdot),u_2(\cdot))=\mathbb{E}\bigg[\int_0^TZ^{u_1,u_2}(t)l_1(t,x^{u_1,u_2}(t),u_1(t),u_2(t))dt+Z^{u_1,u_2}(T)G_1(x^{u_1,u_2}(T))\bigg]
\end{equation}
over $\mathcal{U}_1$, subject to \eqref{sde} and \eqref{dZ}.

We first present the following lemma about some estimates for $x^{u_1,u_2}(\cdot)$ and $Z(\cdot)$, which belong to Li and Tang \cite{LT95}.
\begin{lemma}\label{lemma3.1}
For any $(u_1(\cdot),u_2(\cdot))\in\mathcal{U}_1\times\mathcal{U}_2$, let $x^{u_1,u_2}(\cdot)$ be the corresponding solution to \eqref{sde}. Then there exists some constant $C>0$, such that
\begin{equation}\label{estimate}
\left\{
\begin{aligned}
&\sup_{0\leq t\leq T}\mathbb{E}|x^{u_1,u_2}(t)|^2\leq C\Big[1+\sup_{0\leq t\leq T}\mathbb{E}\Big(|u_1(t)|^2+|u_2(t)|^2\Big)\Big],\\
&\sup_{0\leq t\leq T}\mathbb{E}|Z^{u_1,u_2}(t)|^2\leq C.
\end{aligned}
\right.
\end{equation}
\end{lemma}
The following maximum principle for {\bf Problem of the follower} can be obtained by the classical results in \cite{Ben92} and \cite{LT95}. We omit the details.
\begin{theorem}\label{thm3.1}
Let {\bf(A1)} and {\bf(A2)} hold. For any given $u_2(\cdot)\in\mathcal{U}_2$, if $\bar{u}_1(\cdot)$ is an optimal control of {\bf Problem of the follower}, then the maximum condition
\begin{equation}\label{maximum condition}
\mathbb{E}^{\bar{u}_1,u_2}\big[\big\langle H_{1u_1}(t,x^{\bar{u}_1,u_2},\bar{u}_1,u_2,p,k,K),v_1-\bar{u}_1(t)\big\rangle\big|\mathcal{F}_t^Y\big]\geq0,\ \ \forall v_1\in \mathcal{U}_1,
\end{equation}
holds for a.e. $t\in[0,T]$, $\mathbb{P}^{\bar{u}_1,u_2}$-a.s., where the Hamiltonian function $H_1:[0,T]\times\mathbb{R}^n\times U_1\times U_2\times\mathbb{R}^n
\times\mathbb{R}^{n\times d_1}\times\mathbb{R}^{n\times d_2}\rightarrow\mathbb{R}$ is defined by
\begin{equation}\label{H1}
\begin{aligned}
H_1(t,x^{u_1,u_2},u_1,u_2,p,k,K)&:=\big\langle p(t),b(t,x^{u_1,u_2},u_1,u_2)\big\rangle+{\rm tr}\big\{k(t)^\top\sigma(t,x^{u_1,u_2},u_1,u_2)\big\}\\
&\quad+\big\langle K(t),h(t,x^{u_1,u_2},u_1,u_2)\big\rangle+l_1(t,x^{u_1,u_2},u_1,u_2),
\end{aligned}
\end{equation}
the adjoint process pairs $(P(\cdot),K(\cdot))\in L^2_\mathcal{F}(0,T;\mathbb{R})\times L^2_\mathcal{F}(0,T;\mathbb{R}^{d_1})$ and $(p(\cdot),k(\cdot))\in L^2_\mathcal{F}(0,T;\mathbb{R}^n)\\\times L^2_\mathcal{F}(0,T;\mathbb{R}^{n\times d_1})$ satisfy the following two BSDEs, respectively:
\begin{equation}\label{adjoint1}
\left\{
\begin{aligned}
-dP(t)&=l_1(t,x^{\bar{u}_1,u_2},\bar{u}_1,u_2)dt-K^\top(t)dW^{\bar{u}_1,u_2}(t),\ t\in[0,T],\\
P(T)&=G_1(x^{\bar{u}_1,u_2}(T)),
\end{aligned}
\right.
\end{equation}
\begin{equation}\label{adjoint11}
\left\{
\begin{aligned}
-dp(t)&=\big[l_{1x}(t,x^{\bar{u}_1,u_2},\bar{u}_1,u_2)+h_x^\top(t,x^{\bar{u}_1,u_2},\bar{u}_1,u_2)K^\top(t)+b_x(t,x^{\bar{u}_1,u_2},\bar{u}_1,u_2)p(t)\\
&\quad+\sigma_x(t,x^{\bar{u}_1,u_2},\bar{u}_1,u_2)k(t)\big]dt-k(t)dW(t),\ t\in[0,T],\\
p(T)&=G_{1x}(x^{\bar{u}_1,u_2}(T)).
\end{aligned}
\right.
\end{equation}
\end{theorem}
Then we continue to give the sufficient condition (that is, verification theorem) to guarantee the optimality for control $\bar{u}_1(\cdot)$ of {\bf Problem of the follower}.
\begin{theorem}\label{thm3.2}
Let {\bf(A1)} and {\bf(A2)} hold. For any given $u_2(\cdot)\in\mathcal{U}_2$, let $\bar{u}_1(\cdot)\in\mathcal{U}_1$ and $x^{\bar{u}_1.u_2}(\cdot)$ be the corresponding state. Let $(P(\cdot),K(\cdot))$ and $(p(\cdot),k(\cdot))$ be the adjoint process pairs satisfying \eqref{adjoint1} and \eqref{adjoint11}. Suppose for all $(t,x,u_1,u_2)\in[0,T]\times\mathbb{R}^n\times U_1\times U_2$, $Z^{u_1,u_2}(t)$ is $\mathcal{F}_t^Y$-adapted, maps $(x,u_1)\rightarrow H_1(t,x,u_1,u_2,p,k,K)$ and $x\rightarrow G_1(x)$ are both convex, and
\begin{equation}\label{min}
\mathbb{E}\big[H_1(t,x^{\bar{u}_1,u_2},\bar{u}_1,u_2,p,k,K)\big|\mathcal{F}_t^Y\big]
=\min_{u_1\in\mathcal{U}_1}\mathbb{E}\big[H_1(t,x^{u_1.u_2},u_1,u_2,p,k,K)\big|\mathcal{F}_t^Y\big]
\end{equation}
holds for a.e. $t\in[0,T]$, $\mathbb{P}$-a.s. Then $\bar{u}_1(\cdot)$ is an optimal control of {\bf Problem of the follower}.
\end{theorem}
\begin{proof}
For any $u_1(\cdot)\in\mathcal{U}_1$, we have
\begin{equation}\label{J1-barJ1}
\begin{aligned}
&J_1(u_1(\cdot),u_2(\cdot))-J_1(\bar{u}_1(\cdot),u_2(\cdot))\\
&=\mathbb{E}\bigg[\int_0^T\big[Z^{u_1,u_2}(t)l_1(t,x^{u_1,u_2},u_1,u_2)-Z^{\bar{u}_1,u_2}(t)l_1(t,x^{\bar{u}_1,u_2},\bar{u}_1,u_2)\big]dt\\
&\quad+Z^{u_1,u_2}(T)G_1(x^{u_1,u_2}(T))-Z^{\bar{u}_1,u_2}(T)G_1(x^{\bar{u}_1,u_2}(T))\bigg]\\
&=\mathbb{E}\bigg[\int_0^T\big[\big(Z^{u_1,u_2}(t)-Z^{\bar{u}_1,u_2}(t)\big)l_1(t,x^{\bar{u}_1,u_2},\bar{u}_1,u_2)\big]dt\\
&\qquad+\big(Z^{u_1,u_2}(T)-Z^{\bar{u}_1,u_2}(T)\big)G_1(x^{\bar{u}_1,u_2}(T))\bigg]\\
&\quad+\mathbb{E}\bigg[\int_0^T\big[Z^{u_1,u_2}(t)\big(l_1(t,x^{u_1,u_2},u_1,u_2)-l_1(t,x^{\bar{u}_1,u_2},\bar{u}_1,u_2)\big)\big]dt\\
&\qquad+Z^{u_1,u_2}(T)\big(G_1(x^{u_1,u_2}(T))-G_1(x^{\bar{u}_1,u_2}(T))\big)\bigg]\equiv\rm{I+II}.
\end{aligned}
\end{equation}
Due to the convexity of $G_1(\cdot)$, we have
\begin{equation}\label{II}
\begin{aligned}
\rm{II}&\geq\mathbb{E}^{u_1,u_2}\bigg[\int_0^T\big[l_1(t,x^{u_1,u_2},u_1,u_2)-l_1(t,x^{\bar{u}_1,u_2},\bar{u}_1,u_2)\big]dt\\
&\qquad+G_{1x}(x^{\bar{u}_1,u_2}(T))(x^{u_1,u_2}(T)-x^{\bar{u}_1,u_2}(T))\bigg].
\end{aligned}
\end{equation}
Applying It\^{o}'s formula to $(Z^{u_1,u_2}(\cdot)-Z^{\bar{u}_1,u_2}(\cdot))P(\cdot)$, it is easy to get
\begin{equation}\label{I}
\begin{aligned}
\rm{I}&=\mathbb{E}\bigg[\int_0^TZ^{u_1,u_2}(t)\big[h^{\top}(t,x^{u_1,u_2},u_1,u_2)-h^\top(t,x^{\bar{u}_1,u_2},\bar{u}_1,u_2)\big]K(t)dt\bigg]\\
&=\mathbb{E}^{u_1,u_2}\bigg[\int_0^T\big\langle K(t),h(t,x^{u_1,u_2},u_1,u_2)-h(t,x^{\bar{u}_1,u_2},\bar{u}_1,u_2)\big\rangle dt\bigg].
\end{aligned}
\end{equation}
Similarly, applying It\^{o}'s formula to $p(\cdot)(x^{u_1,u_2}(\cdot)-x^{\bar{u}_1,u_2}(\cdot))$, by \eqref{II}, we have
\begin{equation}\label{II-II}
\begin{aligned}
{\rm II}\geq\mathbb{E}^{u_1,u_2}\bigg[\int_0^T&\big[H_1(t,x^{u_1,u_2},u_1,u_2,p,k,K)-H_1(t,x^{\bar{u}_1,u_2},\bar{u}_1,u_2,p,k,K)\\
&-\big\langle K(t),h(t,x^{u_1.u_2},u_1,u_2)-h(t,x^{\bar{u}_1.u_2},\bar{u}_1,u_2)\big\rangle\\
&-\big\langle H_{1x}(t,x^{\bar{u}_1,u_2},\bar{u}_1,u_2,p,k,K),x^{u_1,u_2}(t)-x^{\bar{u}_1,u_2}(t)\big\rangle\big]dt\bigg].
\end{aligned}
\end{equation}
Using the convexity of $H_1(t,\cdot,\cdot,u_2,p,k,K)$, by \eqref{I} and \eqref{II-II}, we obtain
\begin{equation}\label{I+II}
\begin{aligned}
\rm{I+II}&\geq\mathbb{E}^{u_1,u_2}\bigg[\int_0^T\big\langle H_{1u_1}(t,x^{\bar{u}_1,u_2},\bar{u}_1,u_2,p,k,K),u_1-\bar{u}_1\big\rangle dt\bigg]\\
&=\mathbb{E}\bigg[\int_0^T\mathbb{E}\big[Z^{u_1,u_2}(t)\big\langle H_{1u_1}(t,x^{\bar{u}_1,u_2},\bar{u}_1,u_2,p,k,K),u_1-\bar{u}_1\big\rangle\big|\mathcal{F}_t^Y\big]dt\bigg].
\end{aligned}
\end{equation}
Noticing that $Z^{u_1,u_2}(\cdot)>0$ is $\mathcal{F}_t^Y$-adapted, by the condition \eqref{min}, we get
\begin{equation}\label{ZZZ}
\begin{aligned}
0&\leq\Big\langle\frac{\partial}{\partial u_1}\mathbb{E}\big[H_1(t,x^{\bar{u}_1.u_2},\bar{u}_1,u_2,p,k,K)\big|\mathcal{F}_t^Y\big],u_1-\bar{u}_1\Big\rangle\\
&=\mathbb{E}\big[\big\langle H_{1u_1}(t,x^{\bar{u}_1.u_2},\bar{u}_1,u_2,p,k,K),u_1-\bar{u}_1\big\rangle\big|\mathcal{F}_t^Y\big].
\end{aligned}
\end{equation}
Thus from \eqref{J1-barJ1} we have $J_1(u_1(\cdot),u_2(\cdot))-J_1(\bar{u}_1(\cdot),u_2(\cdot))\geq0$, for any $u_1(\cdot)\in\mathcal{U}_1$. Then we complete our proof.
\end{proof}

\subsection{The problem of the leader}

In this subsection, we firstly state the stochastic optimal control problem with complete information of the leader in detail, then we give the maximum principle and verification theorem for it.
For any $u_2(\cdot)\in\mathcal{U}_2$, by the maximum condition \eqref{maximum condition}, we assume that a functional $\bar{u}_1(t)=\bar{u}_1(t;\hat{x}^{\bar{u}_1,\hat{u}_2},\hat{u}_2,\hat{p},\hat{k},\hat{K})$ is uniquely defined, where we set
\begin{equation}\label{hathat}
\begin{aligned}
&\hat{x}^{\bar{u}_1,\hat{u}_2}(t):=\mathbb{E}^{\bar{u}_1,u_2}\big[x^{\bar{u}_1,u_2}(t)\big|\mathcal{F}_t^Y\big],\ \ \hat{\phi}(t):=\mathbb{E}^{\bar{u}_1,u_2}\big[\phi(t)\big|\mathcal{F}_t^Y\big],\ t\in[0,T],
\end{aligned}
\end{equation}
for $\phi=u_2,p,k,K$. Firstly, the leader encounters the controlled system of FBSDEs:
\begin{equation}\label{leader state0}
\left\{
\begin{aligned}
dx^{\bar{u}_1,u_2}(t)&=b(t,x^{\bar{u}_1,u_2},\bar{u}_1,u_2)dt+\sigma(t,x^{\bar{u}_1,u_2},\bar{u}_1,u_2)dW(t),\\
dZ^{\bar{u}_1,u_2}(t)&=Z^{\bar{u}_1,u_2}(t)h^{\top}(t,x^{\bar{u}_1,u_2},\bar{u}_1,u_2)dY(t),\\
-dp(t)&=\big[l_{1x}(t,x^{\bar{u}_1,u_2},\bar{u}_1,u_2)+h_x^{\top}(t,x^{\bar{u}_1,u_2},\bar{u}_1,u_2)K^\top(t)\\
&\quad+b_x(t,x^{\bar{u}_1,u_2},\bar{u}_1,u_2)p(t)+\sigma_x(t,x^{\bar{u}_1,u_2},\bar{u}_1,u_2)k(t)\big]dt-k(t)dW(t),\\
-dP(t)&=\big[l_1(t,x^{\bar{u}_1,u_2},\bar{u}_1,u_2)+K(t)h(t,x^{\bar{u}_1,u_2},\bar{u}_1,u_2)\big]dt-K(t)dY(t),\ t\in[0,T],\\
  x^{\bar{u}_1,u_2}(0)&=x_0,\ Z^{\bar{u}_1,u_2}(0)=1,\ p(T)=G_{1x}(x^{\bar{u}_1,u_2}(T)),\ P(T)=G_1(x^{\bar{u}_1,u_2}(T)).
\end{aligned}
\right.
\end{equation}
For the simplicity of notations, we denote $x^{u_2}(\cdot)\equiv x^{\bar{u}_1,u_2}(\cdot),Z^{u_2}(\cdot)\equiv Z^{\bar{u}_1,u_2}(\cdot)$ and define $\Phi^L$ on $[0,T]\times\mathbb{R}^n\times U_2$ as
$\Phi^L(t,x^{u_2},u_2):=\Phi(t;x^{\bar{u}_1,u_2},\bar{u}_1(t;\hat{x}^{\bar{u}_1,\hat{u}_2},\hat{u}_2,\hat{p},\hat{k},\hat{K}),u_2)$, for $\Phi=b,\sigma,h,l_1,l_2$. Thus the leader's state equation \eqref{leader state0} is equivalent to:
\begin{equation}\label{leader state1}
\left\{
\begin{aligned}
dx^{u_2}(t)&=b^L(t,x^{u_2},u_2)dt+\sigma^L(t,x^{u_2},u_2)dW(t),\\
dZ^{u_2}(t)&=Z^{u_2}(t)h^L(t,x^{u_2},u_2)^\top dY(t),\\
-dp(t)&=f_1^L(t,x^{u_2},u_2,p,k,K)dt-k(t)dW(t),\\
-dP(t)&=f_2^L(t,x^{u_2},u_2,K)dt-K(t)dY(t),\quad t\in[0,T],\\
  x^{u_2}(0)&=x_0,\ Z^{u_2}(0)=1,\ p(T)=G_{1x}(x^{u_2}(T)),\ P(T)=G_1(x^{u_2}(T)).
\end{aligned}
\right.
\end{equation}
where we define
\begin{equation*}
\begin{aligned}
f_1^L(t,x^{u_2},u_2,p,k,K)&:=l_{1x}^L(t,x^{u_2},u_2)+h_x^L(t,x^{u_2},u_2)^\top K^\top(t)\\
                          &\qquad+b_x^L(t,x^{u_2},u_2)p(t)+\sigma_x^L(t,x^{u_2},u_2)k(t),\\
     f_2^L(t,x^{u_2},u_2,K)&:=l_1^L(t,x^{u_2},u_2)+K(t)h^L(t,x^{u_2},u_2).
\end{aligned}
\end{equation*}

We note that \eqref{leader state1} is a controlled {\it conditional mean-field} FBSDEs, which now is regarded as the ``state" equation of the leader. That is to say, the state of the leader is the six-tuple $(x^{u_2}(\cdot),Z^{u_2}(\cdot),p(\cdot),k(\cdot),P(\cdot),K(\cdot))$. By \eqref{cf2}, we define
\begin{equation}\label{cf22}
\begin{aligned}
&J_2^L(u_2(\cdot)):=J_2(\bar{u}_1(\cdot),u_2(\cdot))\\
&=\mathbb{E}^{\bar{u}_1,u_2}\bigg[\int_0^Tl_2(t,x^{\bar{u}_1,u_2}(t),\bar{u}_1(t),u_2(t))dt+G_2(x^{\bar{u}_1,u_2}(T))\bigg]\\
&\equiv\mathbb{E}^{u_2}\bigg[\int_0^Tl_2^L(t,x^{u_2}(t),u_2(t))dt+G_2(x^{u_2}(T))\bigg]\\
&=\mathbb{E}\bigg[\int_0^TZ^{u_2}(t)l_2^L(t,x^{u_2}(t),u_2(t))dt+Z^{u_2}(T)G_2(x^{u_2}(T))\bigg].
\end{aligned}
\end{equation}

Suppose $\bar{u}_2(\cdot)$ is an optimal control of {\bf Problem of the leader}, and the associated optimal state $(x^{\bar{u}_2}(\cdot),Z^{\bar{u}_2}(\cdot),\bar{p}(\cdot),\bar{k}(\cdot),\bar{P}(\cdot),\bar{K}(\cdot))$ satisfies \eqref{leader state1} with respect to $\bar{u}_2(\cdot)$. In order to derive the maximum principle, we define the perturbed control $u_2^\theta(t):=\bar{u}_2(t)+\theta(v_2(t)-\bar{u}_2(t))$, for $t\in[0,T]$, where $\theta>0$ is sufficiently small and $v_2(\cdot)$ is an arbitrary element of $\mathcal{U}_2$. The convexity of $U_2$ implies that $u_2^\theta(\cdot)\in\mathcal{U}_2$. Let $(x^\theta(\cdot),Z^\theta(\cdot),p^\theta(\cdot),k^\theta(\cdot),P^\theta(\cdot),K^\theta(\cdot))$ be the state corresponding to $u_2^\theta(\cdot)$.
Keeping in mind that $b^L,\sigma^L,h^L,l_1^L,l_2^L$ depend on not only $(x^{\bar u_1,u_2},u_2)$ but also $(\hat{x}^{\bar u_1,\hat u_2},\hat u_2,\hat p,\hat k,\hat K)$. Then we introduce the following system of variational equations whose solution is the six-tuple $(x^1(\cdot),Z^1(\cdot),p^1(\cdot),k^1(\cdot),P^1(\cdot),K^1(\cdot))$:
\begin{equation}\label{variational eq}
\left\{
\begin{aligned}
 dx^1(t)&=\big[\bar{b}_x^Lx^1+\bar{b}_{\hat{x}}^L\hat{x}^1+\bar{b}_{\hat{p}}^L\hat{p}^1+\bar{b}_{\hat{k}}^L\hat{k}^1+\bar{b}_{\hat{K}}^L\hat{K}^1
         +\bar{b}_{u_2}^L(v_2-\bar{u}_2)+\bar{b}_{\hat{u}_2}^L(\hat{v}_2-\hat{\bar{u}}_2)\big]dt\\
        &\quad+\big[\bar{\sigma}_x^Lx^1+\bar{\sigma}_{\hat{x}}^L\hat{x}^1+\bar{\sigma}_{\hat{p}}^L\hat{p}^1
         +\bar{\sigma}_{\hat{k}}^L\hat{k}^1+\bar{\sigma}_{\hat{K}}^L\hat{K}^1+\bar{\sigma}_{u_2}^L(v_2-\bar{u}_2)+\bar{\sigma}_{\hat{u}_2}^L(\hat{v}_2-\hat{\bar{u}}_2)\big]dW(t),\\
 dZ^1(t)&=\big\{\bar{h}^{L\top}Z^1+Z^{\bar{u}_2}\big[\bar{h}_x^{L\top}x^1+\bar{h}_{\hat{x}}^{L\top}\hat{x}^1
         +\bar{h}_{\hat{p}}^{L\top}\hat{p}^1+\bar{h}_{\hat{k}}^{L\top}\hat{k}^1+\bar{h}_{\hat{K}}^{L\top}\hat{K}^1\\
        &\quad+\bar{h}_{u_2}^{L\top}(v_2-\bar{u}_2)+\bar{h}_{\hat{u}_2}^{L\top}(\hat{v}_2-\hat{\bar{u}}_2)\big]\big\}dY(t),\\
-dp^1(t)&=\big[\bar{f}_{1x}^Lx^1+\bar{f}_{1\hat{x}}^L\hat{x}^1+\bar{f}_{1p}^Lp^1+\bar{f}_{1\hat{p}}^L\hat{p}^1+\bar{f}_{1k}^Lk^1+\bar{f}_{1\hat{k}}^L\hat{k}^1
         +\bar{f}_{1K}^LK^1+\bar{f}_{1\hat{K}}^L\hat{K}^1\\
        &\quad+\bar{f}_{1u_2}^L(v_2-\bar{u}_2)+\bar{f}_{1\hat{u}_2}^L(\hat{v}_2-\hat{\bar{u}}_2)\big]dt-k^1(t)dW(t),\\
-dP^1(t)&=\big[\bar{f}_{2x}^Lx^1+\bar{f}_{2\hat{x}}^L\hat{x}^1+\bar{f}_{2\hat{p}}^L\hat{p}^1+\bar{f}_{2\hat{k}}^L\hat{k}^1+\bar{f}_{2K}^LK^1+\bar{f}_{2\hat{K}}^L\hat{K}^1\\
        &\quad+\bar{f}_{2u_2}^L(v_2-\bar{u}_2)+\bar{f}_{2\hat{u}_2}^L(\hat{v}_2-\hat{\bar{u}}_2)\big]dt-K^1(t)dY(t),\quad t\in[0,T],\\
  x^1(0)&=0,\ Z^1(0)=0,\ p^1(T)=G_{1xx}(x^{\bar{u}_2}(T))x^1(T),\ P^1(T)=G_{1x}(x^{\bar{u}_2}(T))x^1(T),
\end{aligned}
\right.
\end{equation}
where we have used $\bar{\Lambda}^L(t)\equiv\Lambda^L(t,x^{\bar{u}_2},\bar{u}_2)$ for $\Lambda=b,\sigma,h,f_1,f_2,l_1,l_2$ and all their partial derivatives.


\begin{remark}
It is necessary for us to analyze the system of variational equations \eqref{variational eq}, which is nontrivial to derive, and we should notice that $\hat{\kappa}_1:=\mathbb{E}^{\bar{u}_2}[\kappa_1|\mathcal{F}_t^Y]$, for $\kappa_1=x^1,Z^1,p^1,k^1,P^1,K^1,\\v_2-\bar{u}_2$. Actually, for example, the part $\bar{b}_{\hat{\kappa}_2}^L\big(\mathbb{E}^{u_2^\theta}[\kappa_2^{u_2^\theta}|\mathcal{F}_t^Y]-\mathbb{E}^{\bar{u}_2}[\kappa_2^{\bar{u}_2}|\mathcal{F}_t^Y]\big)$, for $\kappa_2=x,Z,p,k,P,K$, will appear when we apply the convex variation, which adds difficulty for deduction due to the expectation $\mathbb{E}^{u_2}$ depending on the control variable $u_2$. However, we could overcome this difficulty under some assumptions, by converting it to the expectation $\mathbb{E}$ independent of $u_2$. For this target, we use Bayes' rule to get
\begin{equation}\label{expression}
\begin{aligned}
&\bar{b}_{\hat{\kappa}_2}^L\Big(\mathbb{E}^{u_2^\theta}\big[\kappa_2^{u_2^\theta}\big|\mathcal{F}_t^Y\big]
 -\mathbb{E}^{\bar{u}_2}\big[\kappa_2^{\bar{u}_2}\big|\mathcal{F}_t^Y\big]\Big)\\
&=\bar{b}_{\hat{\kappa}_2}^L\Bigg(\frac{\mathbb{E}\big[Z^{u_2^\theta}\kappa_2^{u_2^\theta}\big|\mathcal{F}_t^Y\big]}{\mathbb{E}\big[Z^{u_2^\theta}\big|\mathcal{F}_t^Y\big]}
 -\frac{\mathbb{E}\big[Z^{\bar{u}_2}\kappa_2^{\bar{u}_2}\big|\mathcal{F}_t^Y\big]}{\mathbb{E}\big[Z^{\bar{u}_2}\big|\mathcal{F}_t^Y\big]}\Bigg)\\
&=\bar{b}_{\hat{\kappa}_2}^L\Bigg(\frac{1}{\mathbb{E}\big[Z^{u_2^\theta}\big|\mathcal{F}_t^Y\big]}\mathbb{E}\big[Z^{u_2^\theta}\kappa_2^{u_2^\theta}
 -Z^{\bar{u}_2}\kappa_2^{\bar{u}_2}\big|\mathcal{F}_t^Y\big]
 +\frac{\mathbb{E}^{\bar{u}_2}\big[\kappa_2^{\bar{u}_2}\big|\mathcal{F}_t^Y\big]}{\mathbb{E}\big[Z^{u_2^\theta}\big|\mathcal{F}_t^Y\big]}
 \mathbb{E}\big[Z^{\bar{u}_2}-Z^{u_2^\theta}\big|\mathcal{F}_t^Y\big]\Bigg).
\end{aligned}
\end{equation}
In the expression \eqref{expression}, $\mathbb{E}\big[Z^{u_2^\theta}\kappa_2^{u_2^\theta}-Z^{\bar{u}_2}\kappa_2^{\bar{u}_2}\big|\mathcal{F}_t^Y\big]$ and $\mathbb{E}\big[Z^{\bar{u}_2}-Z^{u_2^\theta}\big|\mathcal{F}_t^Y\big]$ are just what we want to solve. However, we can not deal with the $\mathbb{E}[Z^{u_2^\theta}|\mathcal{F}_t^Y]$ part. Therefore, we reconsider the assumption in Theorem 3.2 that ``for all $(t,u_2)\in[0,T]\times U_2$, $Z^{u_2}(t)$ is $\mathcal{F}_t^Y$-adapted" (see Theorem 2 of \cite{HLW10}). Then for any $u_2(\cdot)\in\mathcal{U}_2$, Bayes' formula results in
\begin{equation}\label{expression2}
\mathbb{E}^{u_2}\big[\kappa_2^{u_2}(t)\big|\mathcal{F}_t^Y\big]=\mathbb{E}\big[\kappa_2^{u_2}(t)\big|\mathcal{F}_t^Y\big],\quad
\mathbb{E}^{u_2}\big[\kappa_1(t)\big|\mathcal{F}_t^Y\big]=\mathbb{E}\big[\kappa_1(t)\big|\mathcal{F}_t^Y\big],
\end{equation}
and
\begin{equation}\label{expression3}
\bar{b}_{\hat{\kappa}_2}^L\Big(\mathbb{E}^{u_2^\theta}\big[\kappa_2^{u_2^\theta}\big|\mathcal{F}_t^Y\big]
-\mathbb{E}^{\bar{u}_2}\big[\kappa_2^{\bar{u}_2}\big|\mathcal{F}_t^Y\big]\Big)
=\bar{b}_{\hat{\kappa}_2}^L\Big(\mathbb{E}\big[\kappa_2^{u_2^\theta}-\kappa_2^{\bar{u}_2}\big|\mathcal{F}_t^Y\big]\Big),
\end{equation}
for $\kappa_1=x^1,Z^1,p^1,k^1,P^1,K^1,v_2-\bar{u}_2$ and $\kappa_2=x,Z,p,k,P,K$.

Then by means of \eqref{expression2} and \eqref{expression3}, the system of variational equations \eqref{variational eq} can be derived. Similarly, the variational inequality \eqref{variational ineq} and equation \eqref{Gamma} in the following can also be obtained.
\end{remark}


For any $t\in[0,T]$, we set $\tilde{\lambda}^\theta(t)=\frac{\lambda^\theta(t)-\lambda^{\bar{u}_2}(t)}{\theta}-\lambda^1(t)$, for $\lambda=x,Z,p,k,P,K$. By some classical technique (see Wu \cite{Wu10}, Wang et al. \cite{WWX13}), we have the following lemma.
\begin{lemma}
\begin{equation}\label{estimates}
\begin{aligned}
&\lim_{\theta\rightarrow0}\sup_{0\leq t\leq T}\mathbb{E}|\tilde{x}^\theta(t)|^2=0,\quad \lim_{\theta\rightarrow0}\sup_{0\leq t\leq T}\mathbb{E}|\tilde{Z}^\theta(t)|^2=0,\\
&\lim_{\theta\rightarrow0}\sup_{0\leq t\leq T}\mathbb{E}|\tilde{p}^\theta(t)|^2=0,\quad \lim_{\theta\rightarrow0}\sup_{0\leq t\leq T}\mathbb{E}|\tilde{P}^\theta(t)|^2=0,\\
&\lim_{\theta\rightarrow0}\mathbb{E}\int_0^T|\tilde{k}^\theta(t)|^2dt=0,\quad \lim_{\theta\rightarrow0}\mathbb{E}\int_0^T|\tilde{K}^\theta(t)|^2dt=0.
\end{aligned}
\end{equation}
\end{lemma}

Then, we derive the variational inequality. Since $\bar{u}_2(\cdot)$ is an optimal control, we have
\begin{equation}\label{J2^L}
\frac{1}{\theta}\Big[J_2^L(u_2^\theta(\cdot))-J_2^L(\bar{u}_2(\cdot))\Big]\geq0.
\end{equation}
Thus,
\begin{equation*}
\begin{aligned}
&J_2^L(u_2^\theta(\cdot))-J_2^L(\bar{u}_2(\cdot))
=\mathbb{E}\bigg[\int_0^T\big[Z^\theta(t)l_2^L(t,x^\theta,u_2^\theta)-Z^{\bar{u}_2}(t)l_2^L(t,x^{\bar{u}_2},\bar{u}_2)\big]dt\\
&\qquad+Z^\theta(T)G_2(x^\theta(T))-Z^{\bar{u}_2}(T)G_2(x^{\bar{u}_2}(T))\bigg]\\
&=\mathbb{E}\bigg[\int_0^T\Big\{\big(Z^\theta(t)-Z^{\bar{u}_2}(t)\big)\bar{l}_2^L(t)+Z^{\bar{u}_2}(t)\big[\bar{l}_{2x}^L(t)(x^\theta-x^{\bar{u}_2})
+\bar{l}_{2\hat{x}}^L(t)(\hat{x}^\theta-\hat{x}^{\hat{\bar{u}}_2})+\bar{l}_{2\hat{p}}^L(t)(\hat{p}^\theta-\hat{\bar{p}})\\
&\qquad+\bar{l}_{2\hat{k}}^L(t)(\hat{k}^\theta-\hat{\bar{k}})
+\bar{l}_{2\hat{K}}^L(t)(\hat{K}^\theta-\hat{\bar{K}})+\bar{l}_{2u_2}^L(t)(u_2^\theta-\bar{u}_2)+\bar{l}_{2\hat{u}_2}^L(t)(\hat{u}_2^\theta-\hat{\bar{u}}_2)\big]\Big\}dt\\
&\qquad+(Z^\theta(T)-Z^{\bar{u}_2}(T))G_2(x^{\bar{u}_2}(T))+Z^{\bar{u}_2}(T)G_{2x}(x^{\bar{u}_2}(T))(x^\theta(T)-x^{\bar{u}_2}(T))\bigg]\geq0.
\end{aligned}
\end{equation*}
From Lemma 3.2, when $\theta\rightarrow0$, it follows that
\begin{equation*}
\begin{aligned}
&\frac{1}{\theta}\Big[J_2^L(u_2^\theta(\cdot))-J_2^L(\bar{u}_2(\cdot))\Big]
\rightarrow\mathbb{E}\bigg[\int_0^T\Big\{Z^1(t)\bar{l}_2^L(t)+Z^{\bar{u}_2}(t)\big[\bar{l}_{2x}^L(t)x^1+\bar{l}_{2\hat{x}}^L(t)\hat{x}^1\\
&\quad+\bar{l}_{2\hat{p}}^L(t)\hat{p}^1+\bar{l}_{2\hat{k}}^L(t)\hat{k}^1+\bar{l}_{2\hat{K}}^L(t)\hat{K}^1
+\bar{l}_{2u_2}^L(t)(v_2-\bar{u}_2)+\bar{l}_{2\hat{u}_2}^L(t)(\hat{v}_2-\hat{\bar{u}}_2)\big]\Big\}dt\\
&\quad+Z^1(T)G_2(x^{\bar{u}_2}(T))+Z^{\bar{u}_2}(T)G_{2x}(x^{\bar{u}_2}(T))x^1(T)\bigg]\geq0,
\end{aligned}
\end{equation*}
i.e.
\begin{equation}\label{variational ineq}
\begin{aligned}
&\mathbb{E}^{\bar{u}_2}\bigg[\int_0^T\Big[\big(Z^{\bar{u}_2}(t)\big)^{-1}Z^1(t)\bar{l}_2^L(t)+\bar{l}_{2x}^L(t)x^1+\bar{l}_{2\hat{x}}^L(t)\hat{x}^1+\bar{l}_{2\hat{p}}^L(t)\hat{p}^1
+\bar{l}_{2\hat{k}}^L(t)\hat{k}^1\\
&\qquad+\bar{l}_{2\hat{K}}^L(t)\hat{K}^1+\bar{l}_{2u_2}^L(t)(v_2-\bar{u}_2(t))+\bar{l}_{2\hat{u}_2}^L(t)(\hat{v}_2-\hat{\bar{u}}_2(t))\Big]dt\\
&\qquad+\big(Z^{\bar{u}_2}(T)\big)^{-1}Z^1(T)G_2(x^{\bar{u}_2}(T))+G_{2x}(x^{\bar{u}_2}(T))x^1(T)\bigg]\geq0.
\end{aligned}
\end{equation}
Noticing that $(Z^{\bar{u}_2}(\cdot))^{-1}Z^1(\cdot)$ appears in \eqref{variational ineq}, then we set $\Gamma^1(\cdot):=(Z^{\bar{u}_2}(\cdot))^{-1}Z^1(\cdot)$.

Firstly, applying It\^{o}'s formula to $(Z^{\bar{u}_2}(\cdot))^{-1}$, we get
\begin{equation}
d(Z^{\bar{u}_2}(t))^{-1}=-(Z^{\bar{u}_2}(t))^{-1}\bar{h}^{L\top}(t)dY(t)+(Z^{\bar{u}_2}(t))^{-1}\bar{h}^{L\top}(t)\bar{h}^{L}(t)dt.
\end{equation}
Secondly, applying It\^{o}'s formula to $(Z^{\bar{u}_2}(\cdot))^{-1}Z^1(\cdot)$, we obtain
\begin{equation}\label{Gamma}
\left\{
\begin{aligned}
d\Gamma^1(t)&=\big[\bar{h}_{x}^{L\top}(t)x^1(t)+\bar{h}_{\hat{x}}^{L\top}(t)\hat{x}^1(t)+\bar{h}_{\hat{p}}^{L\top}(t)\hat{p}^1(t)
+\bar{h}_{\hat{k}}^{L\top}(t)\hat{k}^1(t)+\bar{h}_{\hat{K}}^{L\top}(t)\hat{K}^1(t)\\
&\quad+\bar{h}_{u_2}^{L\top}(t)(v_2-\bar{u}_2(t))+\bar{h}_{\hat{u}_2}^{L\top}(t)(\hat{v}_2-\hat{\bar{u}}_2(t))\big]dW^{\bar{u}_2}(t),\quad t\in[0,T],\\
\Gamma^1(0)&=0,
\end{aligned}
\right.
\end{equation}
where $W(\cdot)$ and $W^{\bar{u}_2}(\cdot)$ are two independent standard Brownian motions under the probability $\mathbb{P}^{\bar{u}_2}$, $d\mathbb{P}^{\bar{u}_2}:=Z^{\bar{u}_2}(T)d\mathbb{P}$.

Next, we introduce the following system of adjoint equations, consisting of two SDEs and two BSDEs, whose solution is the six-tuple $(q(\cdot),Q(\cdot),\varphi(\cdot),\delta(\cdot),\alpha(\cdot),\beta(\cdot))$:
\begin{equation}\label{adjoint eq}
\left\{
\begin{aligned}
dq(t)&=\Big\{\bar{f}_{1p}^Lq+\mathbb{E}^{\bar{u}_2}\big[\bar{l}_{2\hat{p}}^L+\bar{h}_{\hat{p}}^{L\top}\beta+\bar{b}_{\hat{p}}^L\varphi
+\bar{\sigma}_{\hat{p}}^L\delta+\bar{f}_{2\hat{p}}^LQ+\bar{f}_{1\hat{p}}^Lq\big|\mathcal{F}_t^Y\big]\Big\}dt\\
&\quad+\Big\{\bar{f}_{1k}^Lq+\mathbb{E}^{\bar{u}_2}\big[\bar{l}_{2\hat{k}}^L+\bar{h}_{\hat{k}}^{L\top}\beta+\bar{b}_{\hat{k}}^L\varphi
+\bar{\sigma}_{\hat{k}}^L\delta+\bar{f}_{2\hat{k}}^LQ+\bar{f}_{1\hat{k}}^Lq\big|\mathcal{F}_t^Y\big]\Big\}dW(t),\\
dQ(t)&=\Big\{\bar{f}_{1K}^Lq+\mathbb{E}^{\bar{u}_2}\big[\bar{l}_{2\hat{K}}^L+\bar{h}_{\hat{K}}^{L\top}\beta+\bar{b}_{\hat{K}}^L\varphi
+\bar{\sigma}_{\hat{K}}^L\delta+\bar{f}_{2\hat{K}}^LQ+\bar{f}_{1\hat{K}}^Lq\big|\mathcal{F}_t^Y\big]\Big\}dW^{\bar{u}_2}(t),\\
-d\varphi(t)&=\Big\{\bar{l}_{2x}^L+\bar{h}_x^{L\top}\beta+\bar{b}_x^L\varphi+\bar{\sigma}_x^L\delta+\bar{f}_{2x}^LQ+\bar{f}_{1x}^Lq
+\mathbb{E}^{\bar{u}_2}\big[\bar{l}_{2\hat{x}}^L+\bar{h}_{\hat{x}}^{L\top}\beta+\bar{b}_{\hat{x}}^L\varphi\\
&\quad+\bar{\sigma}_{\hat{x}}^L\delta+\bar{f}_{2\hat{x}}^LQ+\bar{f}_{1\hat{x}}^Lq\big|\mathcal{F}_t^Y\big]\Big\}dt-\delta(t)dW(t),\\
-d\alpha(t)&=\bar{l}_2^L(t)dt-\beta(t)dW^{\bar{u}_2}(t),\quad t\in[0,T],\\
q(0)&=0,\ Q(0)=0,\ \alpha(T)=G_2(x^{\bar{u}_2}(T)),\\
\varphi(T)&=G_{2x}(x^{\bar{u}_2}(T))+G_{1x}(x^{\bar{u}_2}(T))Q(T)+G_{1xx}(x^{\bar{u}_2}(T))q(T).
\end{aligned}
\right.
\end{equation}
Then by the equations \eqref{variational eq}, \eqref{Gamma} and \eqref{adjoint eq}, applying It\^{o}'s formula to $\langle x^1(\cdot),\varphi(\cdot)\rangle-\langle p^1(\cdot),q(\cdot)\rangle\\-\langle P^1(\cdot),Q(\cdot)\rangle+\langle\Gamma^1(\cdot),\alpha(\cdot)\rangle$ on $[0,T]$ and inserting it into \eqref{variational ineq}, we derive
\begin{equation}\label{variational ineq--}
\begin{aligned}
&\mathbb{E}^{\bar{u}_2}\bigg[\int_0^T\Big\{\big\langle \bar{l}_{2u_2}^L+\bar{h}_{u_2}^{L\top}\beta+\bar{b}_{u_2}^L\varphi+\bar{\sigma}_{u_2}^L\delta+\bar{f}_{2u_2}^LQ+\bar{f}_{1u_2}^Lq,v_2-\bar{u}_2\big\rangle\\
&\qquad\quad+\big\langle\mathbb{E}^{\bar{u}_2}\big[\bar{l}_{2\hat{u}_2}^L+\bar{h}_{\hat{u}_2}^{L\top}\beta+\bar{b}_{\hat{u}_2}^L\varphi
+\bar{\sigma}_{\hat{u}_2}^L\delta+\bar{f}_{2\hat{u}_2}^LQ+\bar{f}_{1\hat{u}_2}^Lq\big|\mathcal{F}_t^Y\big],(\hat{v}_2-\hat{\bar{u}}_2)\big\rangle\Big\}dt\bigg]\geq0.
\end{aligned}
\end{equation}
Define the Hamiltonian function $H_2:[0,T]\times\mathbb{R}^n\times U_2\times\mathbb{R}^n\times\mathbb{R}^{n\times d_1}\times\mathbb{R}^{1\times d_2}\times\mathbb{R}^n\times\mathbb{R}\times\mathbb{R}^n\times\mathbb{R}^{n\times d_1}\times\mathbb{R}^{n\times d_2}\rightarrow\mathbb{R}$ as
\begin{equation}\label{H2}
\begin{aligned}
&H_2(t,x^{u_2},u_2,p,k,K;q,Q,\varphi,\delta,\beta):=l_2^L(t,x^{u_2},u_2)+\big\langle\beta(t),h^L(t,x^{u_2},u_2)\big\rangle
+\big\langle\varphi(t),b^L(t,x^{u_2},u_2)\big\rangle\\
&+\big\langle\delta(t),\sigma^L(t,x^{u_2},u_2)\big\rangle+\big\langle Q(t),f_2^L(t,x^{u_2},u_2,K)\big\rangle
+\big\langle q(t),f_1^L(t,x^{u_2},u_2,p,k,K)\big\rangle.
\end{aligned}
\end{equation}
Then the equations \eqref{adjoint eq} is equivalent to:
\begin{equation}\label{adjoint eq2}
\left\{
\begin{aligned}
dq(t)&=\Big\{\bar{H}_{2p}(t)+\mathbb{E}^{\bar{u}_2}\big[\bar{H}_{2\hat{p}}(t)\big|\mathcal{F}_t^Y\big]\Big\}dt
+\Big\{\bar{H}_{2k}(t)+\mathbb{E}^{\bar{u}_2}\big[\bar{H}_{2\hat{k}}(t)\big|\mathcal{F}_t^Y\big]\Big\}dW(t),\\
dQ(t)&=\Big\{\bar{f}_{1K}^L(t)q(t)+\mathbb{E}^{\bar{u}_2}\big[\bar{H}_{2\hat{K}}(t)\big|\mathcal{F}_t^Y\big]\Big\}dW^{\bar{u}_2}(t),\\
-d\varphi(t)&=\Big\{\bar{H}_{2x}(t)+\mathbb{E}^{\bar{u}_2}\big[\bar{H}_{2\hat{x}}(t)\big|\mathcal{F}_t^Y\big]\Big\}dt-\delta(t)dW(t),\\
-d\alpha(t)&=\bar{l}_2^L(t)dt-\beta(t)dW^{\bar{u}_2}(t),\quad t\in[0,T],\\
q(0)&=0,\ Q(0)=0,\ \alpha(T)=G_2(x^{\bar{u}_2}(T)),\\
\varphi(T)&=G_{2x}(x^{\bar{u}_2}(T))+G_{1x}(x^{\bar{u}_2}(T))Q(T)+G_{1xx}(x^{\bar{u}_2}(T))q(T).
\end{aligned}
\right.
\end{equation}
where we set $\bar{H}_{2\lambda}(t)\equiv H_{2\lambda}(t,x^{\bar{u}_2},\bar{u}_2,\bar{p},\bar{k},\bar{K};q,Q,\varphi,\delta,\beta)$ for $\lambda=x,\hat{x},p,\hat{p},k,\hat{k},\hat{K}$.

From (\ref{variational ineq--}) and (\ref{H2}), it is easy to obtain the following maximum principle of the leader.
\begin{theorem}
Let {\bf(A1)} and {\bf(A2)} hold, and $\bar{u}_2(\cdot)\in\mathcal{U}_2$ be an optimal control of {\bf Problem of the leader} and $(x^{\bar{u}_2}(\cdot),Z^{\bar{u}_2}(\cdot),\bar{p}(\cdot),\bar{k}(\cdot),\bar{P}(\cdot),\bar{K}(\cdot))$ be the corresponding optimal state. Let $(q(\cdot),Q(\cdot),\varphi(\cdot),\delta(\cdot),\alpha(\cdot),\beta(\cdot))$ be the adjoint six-tuple satisfying \eqref{adjoint eq2}, then we have
\begin{equation}\label{maximum condition-leader}
\begin{aligned}
&\big\langle H_{2u_2}(t,x^{\bar{u}_2},\bar{u}_2,\bar{p},\bar{k},\bar{K};q,Q,\varphi,\delta,\beta),v_2-\bar{u}_2(t)\big\rangle\\
&+\big\langle \mathbb{E}^{\bar{u}_2}\big[H_{2\hat{u}_2}(t,x^{\bar{u}_2},\bar{u}_2,\bar{p},\bar{k},\bar{K};q,Q,\varphi,\delta,\beta)\big|\mathcal{F}_t^Y\big],
\hat{v}_2-\hat{\bar{u}}_2(t)\big\rangle\geq0,
\end{aligned}
\end{equation}
a.e. $t\in[0,T]$, $\mathbb{P}^{\bar{u}_2}$-a.s. holds for any $v_2\in U_2$.
\end{theorem}

In the following, we wish to establish the verification theorem for {\bf Problem of the leader}. We aim to prove that, under some conditions, for any $v_2(\cdot)\in\mathcal{U}_2$, $J_2^L(v_2(\cdot))-J_2^L(\bar{u}_2(\cdot))\geq0$ holds. However, we find that, during the duality procedure, when applying It\^{o}'s formula, taking integral and expectation, we cannot guarantee that
$$\mathbb{E}^{u_2}\bigg[\int_0^T\Big\{\cdots\Big\}dW^{\bar{u}_2}(t)\bigg]=0$$
holds for any $u_2(\cdot)\in\mathcal{U}_2$ where $\bar{u}_2(\cdot)$ is a candidate optimal control. The reason is that it is not sure that $W^{\bar{u}_2}(\cdot)$ is a Brownian motion under the expectation $\mathbb{E}^{u_2}$. This is the main challenging difficulty which is not easy to solve for us up to now. Therefore, in the following of this paper we consider $h(t,x^{u_1,u_2},u_1,u_2)\equiv h(t)$. In this special case, $Y(\cdot)$ and $W^{u_1,u_2}(\cdot)$ are not controlled by $(u_1(\cdot),u_2(\cdot))$ any more. Thus we could write $W^{u_1,u_2}(\cdot)\equiv\bar{W}(\cdot)$ to be a Brownian motion under the probability $\mathbb{P}^{u_1,u_2}:=\mathbb{\bar{P}}$ directly. Moreover, the adjoint process $(P(\cdot),K(\cdot))$ is needless in the follower's problem, therefore it causes the disappearance of the adjoint processes $(\alpha(\cdot),\beta(\cdot),Q(\cdot))$ in (\ref{adjoint eq2}) of {\bf Problem of the leader}.

In this case, \eqref{maximum condition-leader} in Theorem 3.3 becomes
\begin{equation}
\begin{aligned}
&\big\langle H_{2u_2}(t,x^{\bar{u}_2},\bar{u}_2,\bar{p},\bar{k};q,\varphi,\delta),v_2-\bar{u}_2(t)\big\rangle\\
&+\big\langle\bar{\mathbb{E}}\big[H_{2\hat{u}_2}(t,x^{\bar{u}_2},\bar{u}_2,\bar{p},\bar{k};q,\varphi,\delta)\big|\mathcal{F}_t^Y\big],
\hat{v}_2-\hat{\bar{u}}_2(t)\big\rangle\geq0,
\end{aligned}
\end{equation}
holds for $a.e.\ t\in[0,T]$,\ $\bar{\mathbb{P}}$-$a.s.$, and for any $v_2\in U_2$. Here, the expectation $\bar{\mathbb{E}}$ corresponds to the uncontrolled probability measure $\mathbb{P}^{u_1,u_2}:=\bar{\mathbb{P}}$. The Hamiltonian function \eqref{H2} becomes
\begin{equation}
\begin{aligned}
&H_2(t,x^{u_2},u_2,p,k;q,\varphi,\delta):=\big\langle\varphi(t),b^L(t,x^{u_2},u_2)\big\rangle+\big\langle\delta(t),\sigma^L(t,x^{u_2},u_2)\big\rangle\\
&\quad+\big\langle q(t),f_1^L(t,x^{u_2},u_2,p,k)\big\rangle+l_2^L(t,x^{u_2},u_2),
\end{aligned}
\end{equation}
and the adjoint FBSDE \eqref{adjoint eq2} for $(q(\cdot),\varphi(\cdot),\delta(\cdot))$ reduces to
\begin{equation}\label{adjoint eq22}
\left\{
\begin{aligned}
dq(t)&=\big\{\bar{H}_{2p}(t)+\bar{\mathbb{E}}\big[\bar{H}_{2\hat{p}}(t)\big|\mathcal{F}_t^Y\big]\big\}dt
+\big\{\bar{H}_{2k}(t)+\bar{\mathbb{E}}\big[\bar{H}_{2\hat{k}}(t)\big|\mathcal{F}_t^Y\big]\big\}dW(t),\\
-d\varphi(t)&=\big\{\bar{H}_{2x}(t)+\bar{\mathbb{E}}\big[\bar{H}_{2\hat{x}}(t)\big|\mathcal{F}_t^Y\big]\big\}dt-\delta(t)dW(t),\quad t\in[0,T],\\
q(0)&=0,\ \varphi(T)=G_{2x}(x^{\bar{u}_2}(T))+G_{1xx}(x^{\bar{u}_2}(T))q(T).
\end{aligned}
\right.
\end{equation}

We have the following result. The detailed proof is inspired by Yong and Zhou \cite{YZ99}, by Clarke's generalized gradient. We omit it and let it to the interested readers.
\begin{theorem}
Suppose that {\bf(A1)} and {\bf(A2)} hold. Let $\bar{u}_2(\cdot)\in\mathcal{U}_2$, $(x^{\bar{u}_2}(\cdot),\bar{p}(\cdot),\bar{k}(\cdot))$ be the corresponding state processes and $G_{1x}=Mx$, that is, $G_{1xx}(x)\equiv M\in\mathbb{R}^n$. Let the adjoint equation \eqref{adjoint eq22} admits a solution triple $(q(\cdot),\varphi(\cdot),\delta(\cdot))$ and suppose that $(x^{u_2},u_2,p,k)\rightarrow H_2(t,x^{u_2},u_2,p,k;q,\varphi,\delta)$ and $x\rightarrow G_2(x)$ are convex. Suppose
\begin{equation}\label{minimal condition}
\begin{aligned}
&H_2(t,x^{\bar{u}_2},\bar{u}_2,\bar{p},\bar{k};q,\varphi,\delta)
+\bar{\mathbb{E}}\big[H_2(t,x^{\bar{u}_2},\bar{u}_2,\bar{p},\bar{k};q,\varphi,\delta)\big|\mathcal{F}_t^Y\big]\\
&=\min_{u_2\in\mathcal{U}_2}\Big\{H_2(t,x^{u_2},u_2,p,k;q,\varphi,\delta)+\bar{\mathbb{E}}\big[H_2(t,x^{u_2},u_2,p,k;q,\varphi,\delta)\big|\mathcal{F}_t^Y\big]\Big\}
\end{aligned}
\end{equation}
holds for a.e. $t\in[0,T]$, $\mathbb{P}$-a.s. Then $\bar{u}_2(\cdot)$ is an optimal control of {\bf Problem of the leader}
\end{theorem}

\section{An LQ Stackelberg stochastic differential game with asymmetric noisy observations}

In this section, we deal with an LQ Stackelberg stochastic differential game with asymmetric noisy observations, where the maximum principle and verification theorem developed in the previous section will be useful tools. For notational simplicity, we only consider the case for $n=d_1=d_2=m_1=m_2=1$.

\subsection{The problem of the follower}

Let us consider the following controlled SDE:
\begin{equation}\label{lq sde}
\left\{
\begin{aligned}
dx^{u_1,u_2}(t)&=\big[A(t)x^{u_1,u_2}(t)+B_1(t)u_1(t)+B_2(t)u_2(t)\big]dt\\
&\quad+\big[C(t)x^{u_1,u_2}(t)+D_1(t)u_1(t)+D_2(t)u_2(t)\big]dW(t),\quad t\in[0,T],\\
x^{u_1,u_2}(0)&=x_0,
\end{aligned}
\right.
\end{equation}
and the observation equation:
\begin{equation}\label{observation lq}
\left\{
\begin{aligned}
dY(t)&=h(t)dt+d\bar{W}(t),\quad t\in[0,T],\\
Y(0)&=0,
\end{aligned}
\right.
\end{equation}
where $x_0\in\mathbb{R}$ and $A(\cdot),B_1(\cdot),B_2(\cdot),C(\cdot),D_1(\cdot),D_2(\cdot)$ and $h(\cdot)$ are given deterministic functions. We introduce the following assumption:

{\bf(H1)} $A(\cdot),B_1(\cdot),B_2(\cdot),C(\cdot),D_1(\cdot)$ and $D_2(\cdot)\in L^{\infty}(0,T;\mathbb{R})$.

Firstly, for any chosen $u_2(\cdot)\in\mathcal{U}_2$, the follower would like to choose an $\mathcal{F}_t^Y$-adapted control $\bar{u}_1(\cdot)$ to minimize his cost functional
\begin{equation}\label{cf1 lq}
\begin{aligned}
J_1(u_1(\cdot),u_2(\cdot))=\frac{1}{2}\bar{\mathbb{E}}\bigg[\int_0^T\Big\{Q_1(t)\big|x^{u_1,u_2}(t)\big|^2+R_1(t)\big|u_1(t)\big|^2\Big\}dt+G_1\big|x^{u_1,u_2}(T)\big|^2\bigg],
\end{aligned}
\end{equation}
where the expectation $\bar{\mathbb{E}}$ is corresponding to the probability measure $\mathbb{P}^{u_1,u_2}=\bar{\mathbb{P}}$ under which $W(\cdot)$ and $\bar{W}(\cdot)$ are independent standard Brownian motion mentioned in the previous section. And we also suppose that

{\bf(H2)} $Q_1(\cdot)\geq0,R_1(\cdot)>0$ and $G_1\geq0$ are bounded and deterministic, $R_1^{-1}(\cdot)$ is also bounded.

We write the follower's Hamiltonian function
\begin{equation}\label{H1 lq}
\begin{aligned}
&H_1(t,x,u_1,u_2,p,k)=\big[A(t)x+B_1(t)u_1+B_2(t)u_2\big]p(t)\\
&\quad+\big[C(t)x+D_1(t)u_1+D_2(t)u_2\big]k(t)+\frac{1}{2}Q_1(t)x^2+\frac{1}{2}R_1(t)u_1^2.
\end{aligned}
\end{equation}
From Theorem 3.1, if $\bar{u}_1(\cdot)$ is the optimal control, then we have
\begin{equation}\label{bar u_1}
\bar{u}_1(t)=-R_1^{-1}(t)\big[B_1(t)\hat{p}(t)+D_1(t)\hat{k}(t)\big],\quad a.e.\ t\in[0,T],\ \bar{\mathbb{P}}\mbox{-}a.s.,
\end{equation}
with $\hat{p}(t):=\bar{\mathbb{E}}[p(t)|\mathcal{F}_t^Y]$ and $\hat{k}(t):=\bar{\mathbb{E}}[k(t)|\mathcal{F}_t^Y]$, where $(p(\cdot),k(\cdot))$ is the $\mathcal{F}_t$-adapted solution to the following adjoint BSDE:
\begin{equation}\label{lq adjoint eq1}
\left\{
\begin{aligned}
-dp(t)&=\big[Q_1(t)x^{\bar{u}_1,u_2}(t)+A(t)p(t)+C(t)k(t)\big]dt-k(t)dW(t),\ t\in[0,T],\\
  p(T)&=G_1x^{\bar{u}_1,u_2}(T).
\end{aligned}
\right.
\end{equation}
Noticing that the representation of $\bar{u}_1(\cdot)$ contains the filtering estimate of the second component $k(\cdot)$ of the solution to \eqref{lq adjoint eq1}, since the control variables enter into the diffusion term of \eqref{lq sde}.

Observing the terminal condition in the equation \eqref{lq adjoint eq1}, and the appearance of $u_2(\cdot)$, we set
\begin{equation}\label{relation1}
p(t)=P(t)x^{\bar{u}_1,u_2}(t)+\Theta(t),\quad t\in[0,T],
\end{equation}
for some deterministic and differentiable $\mathbb{R}$-valued function $P(\cdot)$ with $P(T)=G_1$, and $\mathbb{R}$-valued, $\mathcal{F}_t$-adapted process pair $(\Theta(\cdot),\Gamma(\cdot))$ satisfying the BSDE:
\begin{equation}\label{theta beta}
d\Theta(t)=\Xi(t)dt+\Gamma(t)dY(t),\ t\in[0,T],\ \Theta(T)=0.
\end{equation}
In the above equation, $\Xi(\cdot)$ is an $\mathcal{F}_t$-adapted process to be determined later.
Applying It\^{o}'s formula to \eqref{relation1} and noting \eqref{observation lq}, \eqref{theta beta}, we have
\begin{equation}\label{dp}
\begin{aligned}
dp(t)&=\big[\dot{P}x^{\bar{u}_1,u_2}+APx^{\bar{u}_1,u_2}+PB_1\bar{u}_1+PB_2\bar{u}_2+\Xi+\Gamma h\big]dt\\
     &\quad+P\big[Cx^{\bar{u}_1,u_2}+D_1\bar{u}_1+D_2u_2\big]dW(t)+\Gamma d\bar{W}(t).
\end{aligned}
\end{equation}
Comparing \eqref{dp} and \eqref{lq adjoint eq1}, we get
\begin{equation}\label{dt term}
\begin{aligned}
-\big(Q_1x^{\bar{u}_1,u_2}+Ap+Ck\big)&=\dot{P}x^{\bar{u}_1,u_2}+APx^{\bar{u}_1,u_2}+PB_1\bar{u}_1+PB_2u_2+\Xi+\Gamma h,
\end{aligned}
\end{equation}
\begin{equation}\label{dW term}
\begin{aligned}
k=P\big(Cx^{\bar{u}_1,u_2}+D_1\bar{u}_1+D_2u_2\big),\quad \bar{\mathbb{P}}\mbox{-}a.s.,
\end{aligned}
\end{equation}
\begin{equation}\label{d bar W}
\begin{aligned}
\Gamma=0,\quad \bar{\mathbb{P}}\mbox{-}a.s.
\end{aligned}
\end{equation}
Thus \eqref{theta beta} has the unique $\mathcal{F}_t$-adapted solution $(\Theta(\cdot),0)$, which in fact reduces to a {\it backward random differential equation} (BRDE for short).

Substituting \eqref{relation1} and \eqref{dW term} into \eqref{bar u_1}, and supposing that 

{\bf(H3)} $(D_1^2P+R_1)^{-1}$ exist, 

\noindent we obtain
\begin{equation}\label{bar u_11}
\begin{aligned}
\bar{u}_1=&-(D_1^2P+R_1)^{-1}(B_1+D_1C)P\hat{x}^{\bar{u}_1,\hat{u}_2}-(D_1^2P+R_1)^{-1}B_1\hat{\Theta}\\
          &-(D_1^2P+R_1)^{-1}D_1D_2P\hat{u}_2,\quad a.e.\ t\in[0,T],\ \bar{\mathbb{P}}\mbox{-}a.s.,
\end{aligned}
\end{equation}
where $\hat{x}^{\bar{u}_1,\hat{u}_2}(t):=\bar{\mathbb{E}}[x^{\bar{u}_1,u_2}(t)|\mathcal{F}_t^Y]$, $\hat{\Theta}(t):=\bar{\mathbb{E}}[\Theta(t)|\mathcal{F}_t^Y]$ and $\hat{u}_2(t):=\bar{\mathbb{E}}[u_2(t)|\mathcal{F}_t^Y]$.

Inserting \eqref{relation1}, \eqref{dW term} and \eqref{bar u_11} into \eqref{dt term}, we obtain that if the Riccati equation:
\begin{equation}\label{pi}
\left\{
\begin{aligned}
&\dot{P}+2AP+C^2P-(D_1^2P+R_1)^{-1}(B_1+D_1C)^2P^2+Q_1=0,\quad t\in[0,T],\\
&P(T)=G_1,
\end{aligned}
\right.
\end{equation}
admits a unique solution $P(\cdot)$, then we have
\begin{equation}\label{alpha}
\begin{aligned}
\Xi&=-(B_1+D_1C)^2(D_1^2P+R_1)^{-1}P^2x^{\bar{u}_1,u_2}+(B_1+D_1C)^2(D_1^2P+R_1)^{-1}P^2\hat{x}^{\bar{u}_1,\hat{u}_2}\\
      &\quad-(B_2+D_2C)Pu_2+(B_1+D_1C)(D_1^2P+R_1)^{-1}D_1D_2P^2\hat{u}_2\\
      &\quad+(B_1+D_1C)(D_1^2P+R_1)^{-1}B_1P\hat{\Theta}-A\Theta.
\end{aligned}
\end{equation}
With \eqref{alpha}, the BRDE \eqref{theta beta} takes the form
\begin{equation}\label{theta}
\left\{
\begin{aligned}
-d\Theta(t)&=\big[(B_1+D_1C)^2(D_1^2P+R_1)^{-1}P^2\big(x^{\bar{u}_1,u_2}-\hat{x}^{\bar{u}_1,\hat{u}_2}\big)\\
           &\qquad+(B_2+D_2C)Pu_2-(B_1+D_1C)(D_1^2P+R_1)^{-1}D_1D_2P^2\hat{u}_2\\
           &\qquad-(B_1+D_1C)(D_1^2P+R_1)^{-1}B_1P\hat{\Theta}+A\Theta\big]dt,\quad t\in[0,T],\\
\Theta(T)&=0.
\end{aligned}
\right.
\end{equation}
Moreover, for given $u_2(\cdot)$, plugging \eqref{bar u_11} into \eqref{lq sde}, we derive
\begin{equation}
\left\{
\begin{aligned}
dx^{\bar{u}_1,u_2}(t)&=\big[Ax^{\bar{u}_1,u_2}-B_1(D_1^2P+R_1)^{-1}(B_1+D_1C)P\hat{x}^{\bar{u}_1,\hat{u}_2}-(D_1^2P+R_1)^{-1}B_1^2\hat{\Theta}\\
                     &\qquad-B_1(D_1^2P+R_1)^{-1}D_1D_2P\hat{u}_2+B_2u_2\big]dt+\big[Cx^{\bar{u}_1,u_2}\\
                     &\qquad-D_1(D_1^2P+R_1)^{-1}(B_1+D_1C)P\hat{x}^{\bar{u}_1,\hat{u}_2}-D_1(D_1^2P+R_1)^{-1}B_1\hat{\Theta}\\
                     &\qquad-(D_1^2P+R_1)^{-1}D_1^2D_2P\hat{u}_2+D_2u_2\big]dW(t),\quad t\in[0,T],\\
x^{\bar{u}_1,u_2}(0)&=x_0.
\end{aligned}
\right.
\end{equation}
Therefore, from the observation equation \eqref{observation lq} and applying Theorem 8.1 in Lisptser and Shiryayev \cite{LS77}, we can derive the following optimal filtering equation:
\begin{equation}\label{hat x}
\left\{
\begin{aligned}
d\hat{x}^{\bar{u}_1,\hat{u}_2}(t)&=\Big\{\big[A-B_1(D_1^2P+R_1)^{-1}(B_1+D_1C)P\big]\hat{x}^{\bar{u}_1,\hat{u}_2}-(D_1^2P+R_1)^{-1}B_1^2\hat{\Theta}\\
                                 &\qquad+\big[B_2-B_1(D_1^2P+R_1)^{-1}D_1D_2P\big]\hat{u}_2\Big\}dt,\quad t\in[0,T],\\
\hat{x}^{\bar{u}_1,\hat{u}_2}(0)&=x_0,
\end{aligned}
\right.
\end{equation}
which admits a unique $\mathcal{F}_t^Y$-adapted solution $\hat{x}^{\bar{u}_1,\hat{u}_2}(\cdot)$, as long as $\hat\Theta(\cdot)$ is determined.

In fact, similarly, by \eqref{theta} we have
\begin{equation}\label{hat theta}
\left\{
\begin{aligned}
-d\hat{\Theta}(t)&=\Big\{-\big[(B_1+D_1C)(D_1^2P+R_1)^{-1}D_1D_2P^2-(B_2+D_2C)P\big]\hat{u}_2\\
                 &\qquad-\big[(B_1+D_1C)(D_1^2P+R_1)^{-1}B_1P-A\big]\hat{\Theta}\Big\}dt,\quad t\in[0,T],\\
  \hat{\Theta}(T)&=0,
\end{aligned}
\right.
\end{equation}
which admits a unique $\mathcal{F}_t^Y$-adapted solution $\hat{\Theta}(\cdot)$, for given $\hat{u}_2(\cdot)$. Putting \eqref{hat x} and \eqref{hat theta} together, we get the following {\it forward-backward random differential filtering equation} (FBRDFE for short):
\begin{equation}\label{hat x theta}
\left\{
\begin{aligned}
d\hat{x}^{\bar{u}_1,\hat{u}_2}(t)&=\Big\{\big[A-B_1(D_1^2P+R_1)^{-1}(B_1+D_1C)P\big]\hat{x}^{\bar{u}_1,\hat{u}_2}-(D_1^2P+R_1)^{-1}B_1^2\hat{\Theta}\\
                                 &\qquad+\big[B_2-B_1(D_1^2P+R_1)^{-1}D_1D_2P\big]\hat{u}_2\Big\}dt,\\
-d\hat{\Theta}(t)&=\Big\{-\big[(B_1+D_1C)(D_1^2P+R_1)^{-1}D_1D_2P^2-(B_2+D_2C)P\big]\hat{u}_2\\
                 &\qquad-\big[(B_1+D_1C)(D_1^2P+R_1)^{-1}B_1P-A\big]\hat{\Theta}\Big\}dt,\quad t\in[0,T],\\
\hat{x}^{\bar{u}_1,\hat{u}_2}(0)&=x_0,\quad \hat{\Theta}(T)=0,
\end{aligned}
\right.
\end{equation}
which admits a unique $\mathcal{F}_t^Y$-adapted solution $(\hat{x}^{\bar{u}_1,\hat{u}_2}(\cdot),\hat{\Theta}(\cdot))\equiv(\hat{x}^{\bar{u}_1,\hat{u}_2}(\cdot),\hat{\Theta}(\cdot),0)$, for given $\hat{u}_2(\cdot)$.

Now, noting that the conditions in Theorem 3.2 are satisfied, we could summarize the above procedure in the following theorem.
\begin{theorem}
Let {\bf(H1)}-{\bf(H3)} hold, and $P(\cdot)$ satisfies \eqref{pi}. For given $u_2(\cdot)$ of the leader, $\bar{u}_1(\cdot)$ given by \eqref{bar u_11} is the state estimate feedback optimal control of the follower, where $(\hat{x}^{\bar{u}_1,\hat{u}_2}(\cdot),\hat{\Theta}(\cdot))$ is the unique $\mathcal{F}_t^Y$-adapted solution to \eqref{hat x theta}.
\end{theorem}

\subsection{Problem of the leader}

Since the leader knows that the follower will take $\bar{u}_1(\cdot)$ by \eqref{bar u_11}, the state equation of the leader can be written as:
\begin{equation}\label{lq leader state}
\left\{
\begin{aligned}
dx^{u_2}(t)&=\big[Ax^{u_2}-B_1(D_1^2P+R_1)^{-1}(B_1+D_1C)P\hat{x}^{\hat{u}_2}-(D_1^2P+R_1)^{-1}B_1^2\hat{\Theta}\\
           &\qquad-B_1(D_1^2P+R_1)^{-1}D_1D_2P\hat{u}_2+B_2u_2\big]dt\\
           &\quad+\big[Cx^{u_2}-D_1(D_1^2P+R_1)^{-1}(B_1+D_1C)P\hat{x}^{\hat{u}_2}-D_1(D_1^2P+R_1)^{-1}B_1\hat{\Theta}\\
           &\qquad-(D_1^2P+R_1)^{-1}D_1^2D_2P\hat{u}_2+D_2u_2]dW(t),\\
-d\hat{\Theta}(t)&=\Big\{-\big[(B_1+D_1C)(D_1^2P+R_1)^{-1}D_1D_2P^2-(B_2+D_2C)P\big]\hat{u}_2\\
                 &\qquad-\big[(B_1+D_1C)(D_1^2P+R_1)^{-1}B_1P-A\big]\hat{\Theta}\Big\}dt,\quad t\in[0,T],\\
x^{u_2}(0)&=x_0,\quad \hat{\Theta}(T)=0,
\end{aligned}
\right.
\end{equation}
where $x^{u_2}(\cdot):=x^{\bar{u}_1,u_2}(\cdot)$ and $\hat{x}^{\hat{u}_2}(\cdot):=\hat{x}^{\bar{u}_1,\hat{u}_2}(\cdot)$.
Note that \eqref{lq leader state} is a decoupled conditional mean-field FBSDE, and its solvability can be easily obtained.
The leader would like to choose an $\mathcal{F}_t$-adapted control $\bar{u}_2(\cdot)$ to minimize his cost functional
\begin{equation}\label{cf2 lq}
\begin{aligned}
J_2(u_2(\cdot))=\frac{1}{2}\bar{\mathbb{E}}\bigg[\int_0^T\Big\{Q_2(t)\big|x^{u_2}(t)\big|^2+R_2(t)\big|u_2(t)\big|^2\Big\}dt+G_2\big|x^{u_2}(T)\big|^2\bigg].
\end{aligned}
\end{equation}
We suppose

{\bf(H4)}
$Q_2(\cdot)\geq0,R_2(\cdot)>0$ and $G_2\geq0$ are bounded and deterministic, $R_2^{-1}(\cdot)$ is also bounded.

Applying Theorem 3.3 and Theorem 3.4, we can write the leader's Hamiltonian function
\begin{equation}
\begin{aligned}
&H_2(t,x^{u_2},u_2,\hat{\Theta};q,\varphi,\delta)\\
&=\frac{1}{2}Q_2(x^{u_2})^2+\frac{1}{2}R_2u_2^2+\big[Ax^{u_2}-B_1(D_1^2P+R_1)^{-1}(B_1+D_1C)P\hat{x}^{\hat{u}_2}\\
&\quad-(D_1^2P+R_1)^{-1}B_1^2\hat{\Theta}-B_1(D_1^2P+R_1)^{-1}D_1D_2P\hat{u}_2+B_2u_2\big]\varphi\\
&\quad+\big[Cx^{u_2}-D_1(D_1^2P+R_1)^{-1}(B_1+D_1C)P\hat{x}^{\hat{u}_2}-D_1(D_1^2P+R_1)^{-1}B_1\hat{\Theta}\\
&\quad-(D_1^2P+R_1)^{-1}D_1^2D_2P\hat{u}_2+D_2u_2\big]\delta-\Big\{\big[(B_1+D_1C)(D_1^2P+R_1)^{-1}D_1D_2P^2\\
&\quad-(B_2+D_2C)P\big]\hat{u}_2+\big[(B_1+D_1C)(D_1^2P+R_1)^{-1}B_1P-A\big]\hat{\Theta}\Big\}q.
\end{aligned}
\end{equation}
The optimal control $u_2(\cdot)$ of the leader satisfies:
\begin{equation}\label{bar u_2}
\begin{aligned}
&R_2\bar{u}_2+B_2\varphi+D_2\delta-B_1(D_1^2P+R_1)^{-1}D_1D_2P\hat{\varphi}-(D_1^2P+R_1)^{-1}D_1^2D_2P\hat{\delta}\\
&-\big[(B_1+D_1C)(D_1^2P+R_1)^{-1}D_1D_2P^2-(B_2+D_2C)P\big]\hat{q}=0,\quad a.e.\ t\in[0,T],\ \bar{\mathbb{P}}\mbox{-}a.s.,
\end{aligned}
\end{equation}
with $\hat{\varphi}(t):=\bar{\mathbb{E}}[\varphi(t)|\mathcal{F}_t^Y],\hat{\delta}(t):=\bar{\mathbb{E}}[\delta(t)|\mathcal{F}_t^Y]$ and $\hat{q}(t):=\bar{\mathbb{E}}[q(t)|\mathcal{F}_t^Y]$, where $(q(\cdot),\varphi(\cdot),\delta(\cdot))$ satisfies the adjoint FBSDE:
\begin{equation}\label{lq leader adjoint}
\left\{
\begin{aligned}
       dq(t)&=\Big\{-(D_1^2P+R_1)^{-1}B_1^2\varphi-D_1(D_1^2P+R_1)^{-1}B_1\delta\\
            &\qquad-\big[(B_1+D_1C)(D_1^2P+R_1)^{-1}B_1P-A\big]q\Big\}dt,\\
-d\varphi(t)&=\big[Q_2x^{\bar{u}_2}+A\varphi+C\delta-B_1(D_1^2P+R_1)^{-1}(B_1+D_1C)P\hat{\varphi}\\
            &\qquad-D_1(D_1^2P+R_1)^{-1}(B_1+D_1C)P\hat{\delta}\big]dt-\delta(t)dW(t),\quad t\in[0,T],\\
        q(0)&=0,\ \varphi(T)=G_2x^{\bar{u}_2}(T)+G_1q(T).
\end{aligned}
\right.
\end{equation}
Next, for obtaining the state feedback representation of $\bar{u}_2(\cdot)$ via some Riccati equations, let us put \eqref{lq leader state} (for $\bar{u}_2(\cdot)$) and \eqref{lq leader adjoint} together and regard $\begin{pmatrix}x^{\bar{u}_2}(\cdot)\\q(\cdot)\end{pmatrix}$ as the optimal ``state":
\begin{equation}\label{lq leader state2}
\left\{
\begin{aligned}
dx^{\bar{u}_2}(t)&=\big[Ax^{\bar{u}_2}-B_1(D_1^2P+R_1)^{-1}(B_1+D_1C)P\hat{x}^{\hat{\bar{u}}_2}-(D_1^2P+R_1)^{-1}B_1^2\hat{\bar{\Theta}}\\
                 &\qquad-B_1(D_1^2P+R_1)^{-1}D_1D_2P\hat{\bar{u}}_2+B_2\bar{u}_2\big]dt\\
                 &\quad+\big[Cx^{\bar{u}_2}-D_1(D_1^2P+R_1)^{-1}(B_1+D_1C)P\hat{x}^{\hat{\bar{u}}_2}-D_1(D_1^2P+R_1)^{-1}B_1\hat{\bar{\Theta}}\\
                 &\qquad-(D_1^2P+R_1)^{-1}D_1^2D_2P\hat{\bar{u}}_2+D_2\bar{u}_2\big]dW(t),\\
       dq(t)&=\Big\{-(D_1^2P+R_1)^{-1}B_1^2\varphi-D_1(D_1^2P+R_1)^{-1}B_1\delta\\
            &\qquad-\big[(B_1+D_1C)(D_1^2P+R_1)^{-1}B_1P-A\big]q\Big\}dt,\\
-d\varphi(t)&=\big[Q_2x^{\bar{u}_2}+A\varphi+C\delta-B_1(D_1^2P+R_1)^{-1}(B_1+D_1C)P\hat{\varphi}\\
            &\qquad-D_1(D_1^2P+R_1)^{-1}(B_1+D_1C)P\hat{\delta}\big]dt-\delta(t)dW(t),\\
-d\hat{\bar{\Theta}}(t)&=\Big\{-\big[(B_1+D_1C)(D_1^2P+R_1)^{-1}D_1D_2P^2-(B_2+D_2C)P\big]\hat{\bar{u}}_2\\
                 &\qquad-\big[(B_1+D_1C)(D_1^2P+R_1)^{-1}B_1P-A\big]\hat{\bar{\Theta}}\Big\}dt,\quad t\in[0,T],\\
x^{\bar{u}_2}(0)&=x_0,\ q(0)=0,\ \varphi(T)=G_2x^{\bar{u}_2}(T)+G_1q(T),\ \hat{\bar{\Theta}}(T)=0.
\end{aligned}
\right.
\end{equation}
Then, we put
\begin{equation}\label{notation}
X=
\begin{pmatrix}
x^{\bar{u}_2}\\
q
\end{pmatrix},\ \
Y=
\begin{pmatrix}
\varphi\\
\hat{\bar{\Theta}}
\end{pmatrix},\ \
Z=
\begin{pmatrix}
\delta\\
0
\end{pmatrix},\ \
X_0=
\begin{pmatrix}
x_0\\
0
\end{pmatrix},\ \
\bar{G}=
\begin{pmatrix}
G_2&G_1\\
0&0
\end{pmatrix},
\end{equation}
and
\begin{equation*}
\left\{
\begin{aligned}
&\mathcal{A}_1=
\begin{pmatrix}
A&0\\
0&-\big[(B_1+D_1C)(D_1^2P+R_1)^{-1}B_1P-A\big]
\end{pmatrix},\\
&\mathcal{A}_2=
\begin{pmatrix}
-B_1(D_1^2P+R_1)^{-1}(B_1+D_1C)P&0\\
0&0
\end{pmatrix},\ \
\mathcal{A}_3=
\begin{pmatrix}
C&0\\
0&0
\end{pmatrix},\\
&\mathcal{A}_4=
\begin{pmatrix}
-D_1(D_1^2P+R_1)^{-1}(B_1+D_1C)P&0\\
0&0
\end{pmatrix},\ \
\mathcal{A}_5=
\begin{pmatrix}
Q&0\\
0&0
\end{pmatrix},\\
&\mathcal{B}_1=
\begin{pmatrix}
0&-(D_1^2P+R_1)^{-1}B_1^2\\
-(D_1^2P+R_1)^{-1}B_1^2&0
\end{pmatrix},\\
&\mathcal{C}_1=
\begin{pmatrix}
0&0\\
-D_1(D_1^2P+R_1)^{-1}B_1&0
\end{pmatrix},\ \
\mathcal{D}_1=
\begin{pmatrix}
-B_1(D_1^2P+R_1)^{-1}D_1D_2P\\
0
\end{pmatrix},\\
&\mathcal{D}_2=
\begin{pmatrix}
B_2\\
0
\end{pmatrix},\ \
\mathcal{D}_3=
\begin{pmatrix}
-(D_1^2P+R_1)^{-1}D_1^2D_2P\\
0
\end{pmatrix},\ \
\mathcal{D}_4=
\begin{pmatrix}
D_2\\
0
\end{pmatrix},\\
&\mathcal{D}_5=
\begin{pmatrix}
0\\
-\big[(B_1+D_1C)(D_1^2P+R_1)^{-1}D_1D_2P^2-(B_2+D_2C)P\big]
\end{pmatrix}.
\end{aligned}
\right.
\end{equation*}
Then the equation \eqref{lq leader state2} can be rewritten as
\begin{equation}\label{lq leader state3}
\left\{
\begin{aligned}
 dX(t)&=\big[\mathcal{A}_1X+\mathcal{A}_2\hat{X}+\mathcal{B}_1Y+\mathcal{C}_1Z+\mathcal{D}_1\hat{\bar{u}}_2+\mathcal{D}_2\bar{u}_2\big]dt\\
      &\quad+\big[\mathcal{A}_3X+\mathcal{A}_4\hat{X}+\mathcal{C}_1^\top Y+\mathcal{D}_3\hat{\bar{u}}_2+\mathcal{D}_4\bar{u}_2\big]dW(t),\\
-dY(t)&=\big[\mathcal{A}_5X+\mathcal{A}_1Y+\mathcal{A}_2\hat{Y}+\mathcal{A}_3Z+\mathcal{A}_4\hat{Z}+\mathcal{D}_5\hat{\bar{u}}_2\big]dt-Z(t)dW(t),\ t\in[0,T],\\
  X(0)&=X_0,\ Y(T)=\bar{G}X(T),
\end{aligned}
\right.
\end{equation}
and \eqref{bar u_2} can be represented as
\begin{equation}\label{bar u_22}
\begin{aligned}
R_2\bar{u}_2&+\mathcal{D}_2^\top Y+\mathcal{D}_4^\top Z+\mathcal{D}_1^\top\hat{Y}+\mathcal{D}_3^\top\hat{Z}+\mathcal{D}_5^\top\hat{X}=0,\ a.e.\ t\in[0,T],\ \bar{\mathbb{P}}\mbox{-}a.s.
\end{aligned}\end{equation}
Thus we have
\begin{equation}\label{bar u_222}
\bar{u}_2=-R_2^{-1}\big[\mathcal{D}_2^\top Y+\mathcal{D}_4^\top Z+\mathcal{D}_1^\top\hat{Y}+\mathcal{D}_3^\top\hat{Z}+\mathcal{D}_5^\top\hat{X}\big],\ a.e.\ t\in[0,T],\ \bar{\mathbb{P}}\mbox{-}a.s.,
\end{equation}
and
\begin{equation}\label{hat bar u_222}
\hat{\bar{u}}_2=-R_2^{-1}\big[(\mathcal{D}_1+\mathcal{D}_2)^\top\hat{Y}+(\mathcal{D}_3+\mathcal{D}_4)^\top\hat{Z}+\mathcal{D}_5^\top\hat{X}\big],\ a.e.\ t\in[0,T],\ \bar{\mathbb{P}}\mbox{-}a.s.
\end{equation}
Inserting \eqref{bar u_222} and \eqref{hat bar u_222} into \eqref{lq leader state3}, we get
\begin{equation}\label{lq leader state4}
\left\{
\begin{aligned}
 dX(t)&=\Big\{\mathcal{A}_1X+\big[\mathcal{A}_2-(\mathcal{D}_1+\mathcal{D}_2)R_2^{-1}\mathcal{D}_5^\top\big]\hat{X}+(\mathcal{B}_1-\mathcal{D}_2R_2^{-1}\mathcal{D}_2^{\top})Y\\
      &\qquad-\big[\mathcal{D}_1R_2^{-1}(\mathcal{D}_1+\mathcal{D}_2)^\top+\mathcal{D}_2R_2^{-1}\mathcal{D}_1^\top\big]\hat{Y}
       +(\mathcal{C}_1-\mathcal{D}_2R_2^{-1}\mathcal{D}_4^{\top})Z\\
      &\qquad-\big[\mathcal{D}_1R_2^{-1}(\mathcal{D}_3+\mathcal{D}_4)^\top+\mathcal{D}_2R_2^{-1}\mathcal{D}_3^\top\big]\hat{Z}\Big\}dt\\
      &\quad+\Big\{\mathcal{A}_3X+\big[\mathcal{A}_4-(\mathcal{D}_3+\mathcal{D}_4)R_2^{-1}\mathcal{D}_5^\top\big]\hat{X}
       +(\mathcal{C}^\top-\mathcal{D}_4R_2^{-1}\mathcal{D}_2^\top)Y\\
      &\qquad-\big[\mathcal{D}_3R_2^{-1}(\mathcal{D}_1+\mathcal{D}_2)^\top+\mathcal{D}_4R_2^{-1}\mathcal{D}_1^\top\big]\hat{Y}-\mathcal{D}_4R_2^{-1}\mathcal{D}_4^{\top}Z\\
      &\qquad-\big[\mathcal{D}_3R_2^{-1}(\mathcal{D}_3+\mathcal{D}_4)^\top+\mathcal{D}_4R_2^{-1}\mathcal{D}_3^\top\big]\hat{Z}\Big\}dW(t),\\
-dY(t)&=\Big\{\mathcal{A}_5X+\mathcal{A}_1Y+\big[\mathcal{A}_2-\mathcal{D}_5R_2^{-1}(\mathcal{D}_1+\mathcal{D}_2)^\top\big]\hat{Y}+\mathcal{A}_3Z\\
      &\qquad+\big[\mathcal{A}_4-\mathcal{D}_5R_2^{-1}(\mathcal{D}_3+\mathcal{D}_4)^\top\big]\hat{Z}-\mathcal{D}_5R_2^{-1}\mathcal{D}_5^{\top}\hat{X}\Big\}dt-ZdW(t),\ t\in[0,T],\\
  X(0)&=X_0,\ Y(T)=\bar{G}X(T).
\end{aligned}
\right.
\end{equation}
In order to decouple the conditional mean-field system \eqref{lq leader state4}, we set
\begin{equation}\label{relation2}
Y(t)=\Pi_1(t)X(t)+\Pi_2(t)\hat{X}(t),
\end{equation}
where $\Pi_1(\cdot)$ and $\Pi_2(\cdot)$ are both differentiable, deterministic matrix-valued functions with $\Pi_1(T)=\bar{G}$ and $\Pi_2(T)=0$.

First, from the forward equation of \eqref{lq leader state4}, applying again Theorem 8.1 in \cite{LS77}, we obtain
\begin{equation}\label{hat X}
\left\{
\begin{aligned}
d\hat{X}(t)&=\Big\{\big[\mathcal{A}_1+\mathcal{A}_2-(\mathcal{D}_1+\mathcal{D}_2)R_2^{-1}\mathcal{D}_5^\top\big]\hat{X}
            +\big[\mathcal{B}_1-(\mathcal{D}_1+\mathcal{D}_2)R_2^{-1}(\mathcal{D}_1+\mathcal{D}_2)^\top\big]\hat{Y}\\
           &\quad+\big[\mathcal{C}_1-(\mathcal{D}_1+\mathcal{D}_2)R_2^{-1}(\mathcal{D}_3+\mathcal{D}_4)^\top\big]\hat{Z}\Big\}dt,\ t\in[0,T],\\
 \hat{X}(0)&=X_0.
\end{aligned}
\right.
\end{equation}
Applying It\^{o}'s formula to \eqref{relation2}, we have
\begin{equation}\label{dY}
\begin{aligned}
dY(t)&=\Big\{\dot{\Pi}_1X+\Pi_1\mathcal{A}_1X+\Pi_1\big[\mathcal{A}_2-(\mathcal{D}_1+\mathcal{D}_2)R_2^{-1}\mathcal{D}_5^\top\big]\hat{X}
      +\Pi_1(\mathcal{B}_1-\mathcal{D}_2R_2^{-1}\mathcal{D}_2^\top)Y\\
     &\qquad-\Pi_1\big[\mathcal{D}_1R_2^{-1}(\mathcal{D}_1+\mathcal{D}_2)^\top+\mathcal{D}_2R_2^{-1}\mathcal{D}_1^\top\big]\hat{Y}
      +\Pi_1(\mathcal{C}_1-\mathcal{D}_2R_2^{-1}\mathcal{D}_4^\top)Z\\
     &\qquad-\Pi_1\big[\mathcal{D}_1R_2^{-1}(\mathcal{D}_3+\mathcal{D}_4)^\top+\mathcal{D}_2R_2^{-1}\mathcal{D}_3^\top\big]\hat{Z}+\dot{\Pi}_2\hat{X}
      +\Pi_2\big[\mathcal{A}_1+\mathcal{A}_2\\
     &\qquad-(\mathcal{D}_1+\mathcal{D}_2)R_2^{-1}\mathcal{D}_5^\top\big]\hat{X}
      +\Pi_2\big[\mathcal{B}_1-(\mathcal{D}_1+\mathcal{D}_2)R_2^{-1}(\mathcal{D}_1+\mathcal{D}_2)^\top\big]\hat{Y}\\
     &\qquad+\Pi_2\big[\mathcal{C}_1-(\mathcal{D}_1+\mathcal{D}_2)R_2^{-1}(\mathcal{D}_3+\mathcal{D}_4)^\top\big]\hat{Z}\Big\}dt\\
     &\quad+\Big\{\Pi_1\mathcal{A}_3X+\Pi_1\big[\mathcal{A}_4-(\mathcal{D}_3+\mathcal{D}_4)R_2^{-1}\mathcal{D}_5^\top\big]\hat{X}
      +\Pi_1(\mathcal{C}_1^\top-\mathcal{D}_4R_2^{-1}\mathcal{D}_2^\top)Y\\
     &\qquad-\Pi_1\big[\mathcal{D}_3R_2^{-1}(\mathcal{D}_1+\mathcal{D}_2)^\top+\mathcal{D}_4R_2^{-1}\mathcal{D}_1^\top\big]\hat{Y}-\Pi_1\mathcal{D}_4R_2^{-1}\mathcal{D}_4^{\top}Z\\
     &\qquad-\Pi_1\big[\mathcal{D}_3R_2^{-1}(\mathcal{D}_3+\mathcal{D}_4)^\top+\mathcal{D}_4R_2^{-1}\mathcal{D}_3^\top\big]\hat{Z}\Big\}dW(t).
\end{aligned}
\end{equation}
Comparing the diffusion term between the BSDE in \eqref{lq leader state4} and \eqref{dY}, it yields
\begin{equation}\label{dW term2}
\begin{aligned}
Z&=\Pi_1\mathcal{A}_3X+\Pi_1\big[\mathcal{A}_4-(\mathcal{D}_3+\mathcal{D}_4)R_2^{-1}\mathcal{D}_5^\top\big]\hat{X}
  +\Pi_1(\mathcal{C}_1^\top-\mathcal{D}_4R_2^{-1}\mathcal{D}_2^\top)Y\\
 &\quad-\Pi_1\big[\mathcal{D}_3R_2^{-1}(\mathcal{D}_1+\mathcal{D}_2)^\top+\mathcal{D}_4R_2^{-1}\mathcal{D}_1^\top\big]\hat{Y}-\Pi_1\mathcal{D}_4R_2^{-1}\mathcal{D}_4^\top Z\\
 &\quad-\Pi_1\big[\mathcal{D}_3R_2^{-1}(\mathcal{D}_3+\mathcal{D}_4)^\top+\mathcal{D}_4R_2^{-1}\mathcal{D}_3^\top\big]\hat{Z},\quad \bar{\mathbb{P}}\mbox{-}a.s.
\end{aligned}
\end{equation}
Taking $\bar{\mathbb{E}}[\cdot|\mathcal{F}_t^Y]$ on both sides of \eqref{dW term2}, and supposing that 

{\bf(H5)} $\mathcal{M}_1:=\big[I+\Pi_1(\mathcal{D}_3+\mathcal{D}_4)R_2^{-1}(\mathcal{D}_3+\mathcal{D}_4)^\top\big]^{-1}$ exists, 

\noindent we have
\begin{equation}\label{hat Z}
\hat{Z}=\Sigma_1(\Pi_1,\Pi_2)\hat{X},\quad \bar{\mathbb{P}}\mbox{-}a.s.,
\end{equation}
where
\begin{equation*}
\begin{aligned}
\Sigma_1(\Pi_1,\Pi_2)&:=\mathcal{M}_1\Big\{\Pi_1\big[\mathcal{A}_3+\mathcal{A}_4-(\mathcal{D}_3+\mathcal{D}_4)R_2^{-1}\mathcal{D}_5^\top\big]\\
                     &\qquad+\Pi_1\big[\mathcal{C}_1^\top-(\mathcal{D}_3+\mathcal{D}_4)R_2^{-1}(\mathcal{D}_1+\mathcal{D}_2)^\top\big](\Pi_1+\Pi_2)\Big\}.
\end{aligned}
\end{equation*}
Then putting \eqref{hat Z} back into \eqref{dW term2}, and supposing that 

{\bf(H6)} $\mathcal{M}_2:=\big[I+\Pi_1\mathcal{D}_4R_2^{-1}\mathcal{D}_4^\top\big]^{-1}$ exists, 

\noindent we get
\begin{equation}\label{Z}
Z=\Sigma_2(\Pi_1)X+\Sigma_3(\Pi_1,\Pi_2)\hat{X},\quad \bar{\mathbb{P}}\mbox{-}a.s.,
\end{equation}
where
\begin{equation*}
\begin{aligned}
      \Sigma_2(\Pi_1)&:=\mathcal{M}_2\big[\Pi_1\mathcal{A}_3+\Pi_1(\mathcal{C}_1^\top-\mathcal{D}_4R_2^{-1}\mathcal{D}_2^\top)\Pi_1\big],\\
\Sigma_3(\Pi_1,\Pi_2)&:=\mathcal{M}_2\Big\{\Pi_1\big[\mathcal{A}_4-(\mathcal{D}_3+\mathcal{D}_4)R_2^{-1}\mathcal{D}_5^\top\big]
                      +\Pi_1(\mathcal{C}_1^\top-\mathcal{D}_4R_2^{-1}\mathcal{D}_2^\top)\Pi_2\\
                     &\qquad-\Pi_1\big[\mathcal{D}_3R_2^{-1}(\mathcal{D}_1+\mathcal{D}_2)^\top+\mathcal{D}_4R_2^{-1}\mathcal{D}_1^\top\big](\Pi_1+\Pi_2)\\
                     &\qquad-\Pi_1\big[\mathcal{D}_3R_2^{-1}(\mathcal{D}_3+\mathcal{D}_4)^\top+\mathcal{D}_4R_2^{-1}\mathcal{D}_3^\top\big]\Sigma_1(\Pi_1,\Pi_2)\Big\}.
\end{aligned}
\end{equation*}
Next, comparing the drift term between the BSDE in \eqref{lq leader state4} and \eqref{dY}, it leads to
\begin{equation}\label{dt term2}
\begin{aligned}
&\dot{\Pi}_1X+\Pi_1\mathcal{A}_1X+\Pi_1\big[\mathcal{A}_2-(\mathcal{D}_1+\mathcal{D}_2)R_2^{-1}\mathcal{D}_5^\top\big]\hat{X}
 +\Pi_1(\mathcal{B}_1-\mathcal{D}_2R_2^{-1}\mathcal{D}_2^\top)Y\\
&-\Pi_1\big[\mathcal{D}_1R_2^{-1}(\mathcal{D}_1+\mathcal{D}_2)^\top+\mathcal{D}_2R_2^{-1}\mathcal{D}_1^\top\big]\hat{Y}
 +\Pi_1(\mathcal{C}_1-\mathcal{D}_2R_2^{-1}\mathcal{D}_4^\top)Z\\
&-\Pi_1\big[\mathcal{D}_1R_2^{-1}(\mathcal{D}_3+\mathcal{D}_4)^\top+\mathcal{D}_2R_2^{-1}\mathcal{D}_3^\top\big]\hat{Z}
 +\dot{\Pi}_2\hat{X}+\Pi_2\big[\mathcal{A}_1+\mathcal{A}_2-(\mathcal{D}_1+\mathcal{D}_2)R_2^{-1}\mathcal{D}_5^\top\big]\hat{X}\\
&+\Pi_2\big[\mathcal{B}_1-(\mathcal{D}_1+\mathcal{D}_2)R_2^{-1}(\mathcal{D}_1+\mathcal{D}_2)^\top\big]\hat{Y}
 +\Pi_2\big[\mathcal{C}_1-(\mathcal{D}_1+\mathcal{D}_2)R_2^{-1}(\mathcal{D}_3+\mathcal{D}_4)^\top\big]\hat{Z}\\
&+\mathcal{A}_5X+\mathcal{A}_1Y+\big[\mathcal{A}_2-\mathcal{D}_5R_2^{-1}(\mathcal{D}_1+\mathcal{D}_2)^\top\big]\hat{Y}+\mathcal{A}_3Z
+\big[\mathcal{A}_4-\mathcal{D}_5R_2^{-1}(\mathcal{D}_3+\mathcal{D}_4)^\top\big]\hat{Z}\\
&-\mathcal{D}_5R_2^{-1}\mathcal{D}_5^\top\hat{X}=0.
\end{aligned}
\end{equation}
After inserting \eqref{relation2}, \eqref{hat Z} and \eqref{Z} into \eqref{dt term2}, we derive the following two Riccati equations:
\begin{equation}\label{R1}
\left\{
\begin{aligned}
&\dot{\Pi}_1+\Pi_1\mathcal{A}_1+\mathcal{A}_1\Pi_1+\Pi_1(\mathcal{B}_1-\mathcal{D}_2R_2^{-1}\mathcal{D}_2^\top)\Pi_1+\mathcal{A}_5
 +\big[\mathcal{A}_3+\Pi_1(\mathcal{C}_1-\mathcal{D}_2R_2^{-1}\mathcal{D}_4^\top)\big]\\
&\quad\times\mathcal{M}_2\Pi_1\big[\mathcal{A}_3+(\mathcal{C}_1^\top-\mathcal{D}_4R_2^{-1}\mathcal{D}_2^\top)\Pi_1\big]=0,\ t\in[0,T],\\
&\Pi_1(T)=\bar{G},
\end{aligned}
\right.
\end{equation}
\begin{equation}\label{R2}
\left\{
\begin{aligned}
&\dot{\Pi}_2+(\Pi_1+\Pi_2)\big[\mathcal{A}_2-(\mathcal{D}_1+\mathcal{D}_2)R_2^{-1}\mathcal{D}_5^\top\big]
 +\big[\mathcal{A}_2-\mathcal{D}_5R_2^{-1}(\mathcal{D}_1+\mathcal{D}_2)^\top\big](\Pi_1+\Pi_2)\\
&\quad+\Pi_2\mathcal{A}_1+\mathcal{A}_1\Pi_2+(\Pi_1+\Pi_2)\big[\mathcal{B}_1-(\mathcal{D}_1+\mathcal{D}_2)R_2^{-1}(\mathcal{D}_1+\mathcal{D}_2)^\top\big](\Pi_1+\Pi_2)\\
&\quad-\Pi_1(\mathcal{B}_1-\mathcal{D}_2R_2^{-1}\mathcal{D}_2^{\top})\Pi_1
 +\big[\mathcal{A}_3+\Pi_1(\mathcal{C}_1-\mathcal{D}_2R_2^{-1}\mathcal{D}_4^\top)\big]\Sigma_3(\Pi_1,\Pi_2)\\
&\quad+\Big\{\big[\mathcal{A}_4-\mathcal{D}_5R_2^{-1}(\mathcal{D}_3+\mathcal{D}_4)^\top\big]
 +\Pi_2\big[\mathcal{C}_1-(\mathcal{D}_1+\mathcal{D}_2)R_2^{-1}(\mathcal{D}_3+\mathcal{D}_4)^\top\big]\\
&\quad-\Pi_1\big[\mathcal{D}_1R_2^{-1}(\mathcal{D}_3+\mathcal{D}_4)^\top+\mathcal{D}_2R_2^{-1}\mathcal{D}_3^\top\big]\Big\}\Sigma_1(\Pi_1,\Pi_2)
 -\mathcal{D}_5R_2^{-1}\mathcal{D}_5^\top=0,\ t\in[0,T],\\
&\Pi_2(T)=0.
\end{aligned}
\right.
\end{equation}

\begin{remark}
The two Riccati equations \eqref{R1} and \eqref{R2} are not standard and entirely new, and we cannot obtain their solvability up to now. However, a special case could be dealt with by some existing reuslts.
\end{remark}

We consider the case that $D_1=D_2=0$, then we have $\mathcal{A}_4=\mathcal{C}_1=\mathcal{D}_1=\mathcal{D}_3=\mathcal{D}_4=0$. The Riccati equations \eqref{R1} and \eqref{R2} of $\Pi_1(\cdot)$ and $\Pi_2(\cdot)$ reduce to:
\begin{equation}\label{R11}
\left\{
\begin{aligned}
&\dot{\Pi}_1+\Pi_1\mathcal{A}_1+\mathcal{A}_1\Pi_1+\Pi_1(\mathcal{B}_1-\mathcal{D}_2R_2^{-1}\mathcal{D}_2^\top)\Pi_1+\mathcal{A}_3\Pi_1\mathcal{A}_3+\mathcal{A}_5=0,\ t\in[0,T],\\
&\Pi_1(T)=\bar{G},
\end{aligned}
\right.
\end{equation}
\begin{equation}\label{R22}
\left\{
\begin{aligned}
&\dot{\Pi}_2+\Pi_2(\mathcal{A}_1+\mathcal{A}_2-\mathcal{D}_2R_2^{-1}\mathcal{D}_5^\top)
 +(\mathcal{A}_1+\mathcal{A}_2-\mathcal{D}_5R_2^{-1}\mathcal{D}_2^\top)\Pi_2-\mathcal{D}_5R_2^{-1}\mathcal{D}_5^\top\\
&\quad+\Pi_1(\mathcal{A}_2-\mathcal{D}_2R_2^{-1}\mathcal{D}_5^\top)+(\mathcal{A}_2-\mathcal{D}_5R_2^{-1}\mathcal{D}_2^\top)\Pi_1
 +\Pi_2(\mathcal{B}_1-\mathcal{D}_2R_2^{-1}\mathcal{D}_2^\top)\Pi_2\\
&\quad+\Pi_1(\mathcal{B}_1-\mathcal{D}_2R_2^{-1}\mathcal{D}_2^\top)\Pi_2+\Pi_2(\mathcal{B}_1-\mathcal{D}_2R_2^{-1}\mathcal{D}_2^\top)\Pi_1=0,\ t\in[0,T],\\
&\Pi_2(T)=0,
\end{aligned}
\right.
\end{equation}
respectively. The solvability of \eqref{R11} and \eqref{R22} can be guaranteed by the sufficient conditions in Chapter 6 of Yong and Zhou \cite{YZ99} and Theorem 5.3 of Yong \cite{Yong99}. We omit the details.

Substituting \eqref{relation2}, \eqref{hat Z} and \eqref{Z} into \eqref{bar u_222}, we get
\begin{equation}\label{bar u_2222}
\begin{aligned}
\bar{u}_2=&-R_2^{-1}\Big\{\big[\mathcal{D}_2^\top\Pi_1+\mathcal{D}_4^\top\Sigma_2(\Pi_1)\big]X
           +\big[\mathcal{D}_1^\top(\Pi_1+\Pi_2)+\mathcal{D}_2^\top\Pi_2+\mathcal{D}_3^\top\Sigma_1(\Pi_1,\Pi_2)\\
          &\quad+\mathcal{D}_4^\top\Sigma_3(\Pi_1,\Pi_2)+\mathcal{D}_5^\top\big]\hat{X}\Big\},\ a.e.\ t\in[0,T],\ \bar{\mathbb{P}}\mbox{-}a.s.,
\end{aligned}
\end{equation}
where the optimal ``state" $X(\cdot)$ and its optimal estimate $\hat{X}(\cdot)$ satisfy
\begin{equation}\label{dX}
\left\{
\begin{aligned}
dX(t)&=\Big\{\big[\mathcal{A}_1+(\mathcal{B}_1-\mathcal{D}_2R_2^{-1}\mathcal{D}_2^\top)\Pi_1+(\mathcal{C}_1-\mathcal{D}_2R_2^{-1}\mathcal{D}_4^\top)\Sigma_2(\Pi_1)\big]X
      +\Big[\mathcal{A}_2-(\mathcal{D}_1+\mathcal{D}_2)R_2^{-1}\\
     &\quad\ \times\mathcal{D}_5^\top+(\mathcal{B}_1-\mathcal{D}_2R_2^{-1}\mathcal{D}_2^\top)\Pi_2
      -\big[\mathcal{D}_1R_2^{-1}(\mathcal{D}_1+\mathcal{D}_2)^\top+\mathcal{D}_2R_2^{-1}\mathcal{D}_1^\top\big](\Pi_1+\Pi_2)+(\mathcal{C}_1\\
     &\quad\ -\mathcal{D}_2R_2^{-1}\mathcal{D}_4^\top)\Sigma_3(\Pi_1,\Pi_2)
      -\big[\mathcal{D}_1R_2^{-1}(\mathcal{D}_3+\mathcal{D}_4)^\top+\mathcal{D}_2R_2^{-1}\mathcal{D}_3^\top\big]\Sigma_1(\Pi_1,\Pi_2)\Big]\hat{X}\Big\}dt\\
     &\quad+\Big\{\big[\mathcal{A}_3+(\mathcal{C}_1^\top-\mathcal{D}_4R_2^{-1}\mathcal{D}_2^\top)\Pi_1-\mathcal{D}_4R_2^{-1}\mathcal{D}_4^\top\Sigma_2(\Pi_1)\big]X
      +\Big[\mathcal{A}_4-(\mathcal{D}_3+\mathcal{D}_4)R_2^{-1}\mathcal{D}_5^\top\\
     &\quad\ +(\mathcal{C}_1^\top-\mathcal{D}_4R_2^{-1}\mathcal{D}_2^\top)\Pi_2-\big[\mathcal{D}_3R_2^{-1}(\mathcal{D}_1+\mathcal{D}_2)^\top
      +\mathcal{D}_4R_2^{-1}\mathcal{D}_1^\top\big](\Pi_1+\Pi_2)-\mathcal{D}_4R_2^{-1}\mathcal{D}_4^\top\\
     &\quad\ \times\Sigma_3(\Pi_1,\Pi_2)-\big[\mathcal{D}_3R_2^{-1}(\mathcal{D}_3+\mathcal{D}_4)^\top
      +\mathcal{D}_4R_2^{-1}\mathcal{D}_3^\top\big]\Sigma_1(\Pi_1,\Pi_2)\Big]\hat{X}\Big\}dW(t),\ t\in[0,T],\\
X(0)&=X_0,
\end{aligned}
\right.
\end{equation}
and
\begin{equation}\label{d hatX}
\left\{
\begin{aligned}
d\hat{X}(t)&=\Big\{\mathcal{A}_1+\mathcal{A}_2+\big[\mathcal{B}_1-(\mathcal{D}_1+\mathcal{D}_2)R_2^{-1}(\mathcal{D}_1+\mathcal{D}_2)^\top\big](\Pi_1+\Pi_2)\\
           &\qquad-(\mathcal{D}_1+\mathcal{D}_2)R_2^{-1}\mathcal{D}_5^\top+(\mathcal{C}_1-\mathcal{D}_2R_2^{-1}\mathcal{D}_4^\top)\Sigma_2(\Pi_1)\\
           &\qquad-\big[\mathcal{D}_1R_2^{-1}(\mathcal{D}_3+\mathcal{D}_4)^\top+\mathcal{D}_2R_2^{-1}\mathcal{D}_3^\top\big]\Sigma_1(\Pi_1,\Pi_2)\\
           &\qquad+(\mathcal{C}_1-\mathcal{D}_2R_2^{-1}\mathcal{D}_4^{\top})\Sigma_3(\Pi_1,\Pi_2)\Big\}\hat{X}dt,\ t\in[0,T],\\
 \hat{X}(0)&=X_0,
\end{aligned}
\right.
\end{equation}
respectively. We summarize the above in the following theorem.
\begin{theorem}
Let {\bf(H1)-(H6)} hold, $\Pi_1(\cdot)$ and $\Pi_2(\cdot)$ satisfy \eqref{R1} and \eqref{R2}, respectively, $X(\cdot)$ be the $\mathcal{F}_t$-adapted solution to \eqref{dX}, and $\hat{X}(\cdot)$ be the $\mathcal{F}_t^Y$-adapted solution to \eqref{d hatX}. Define $Y(\cdot),Z(\cdot)$ and $\hat{Z}(\cdot)$ by \eqref{relation2}, \eqref{Z} and \eqref{hat Z}, respectively. Then equation \eqref{lq leader state4} holds and $\bar{u}_2(\cdot)$ given by \eqref{bar u_2222} is the state estimate feedback representation of the leader's optimal control.
\end{theorem}

Finally, the optimal control $\bar{u}_1(\cdot)$ of the follower can also be represented in $\hat{X}(\cdot)$. More precisely, by \eqref{bar u_11}, noting \eqref{notation} and \eqref{bar u_222}, we derive
\begin{equation}\label{bar u_111}
\begin{aligned}
\bar{u}_1&=\Big\{\mathcal{A}_6+\mathcal{B}_2(\Pi_1+\Pi_2)+(D_1^2P+R_1)^{-1}D_1D_2PR_2^{-1}\big[(\mathcal{D}_1+\mathcal{D}_2)^\top(\Pi_1+\Pi_2)\\
         &\quad+(\mathcal{D}_3+\mathcal{D}_4)^\top\Sigma_1(\Pi_1,\Pi_2)+\mathcal{D}_5^\top\big]\Big\}\hat{X},
          \ a.e.\ t\in[0,T],\ \bar{\mathbb{P}}\mbox{-}a.s.,
\end{aligned}
\end{equation}
where $\mathcal{A}_6:=\begin{pmatrix}-(D_1^2P+R_1)^{-1}(B_1+D_1C)P&0\end{pmatrix}$ and $\mathcal{B}_2:=\begin{pmatrix}0&-(D_1^2P+R_1)^{-1}B_1\end{pmatrix}$.

Thus, the open-loop Stackelberg equilibrium $(\bar u_1(\cdot),\bar u_2(\cdot))$ is given by \eqref{bar u_111} and \eqref{bar u_2222}, in its state estimate feedback form.

\section{Concluding remarks}

In this paper, we have discussed the Stackelberg stochastic differential game with asymmetric noisy observation. This kind of game problem has three interesting characteristics worthy of being emphasized. Firstly, the follower could only observe the noisy observation process, while the leader can observe both the state and noisy observation processes. Thus, the information between the follower and the leader has the asymmetric feature. Second, the leader's problem is solved under some mild assumption, with the aid of some new Riccati equations. Finally, the optimal control of the leader relies not only on the state but also on its estimate based on the observation process.

Possible extension to the Stackelberg stochastic differential game with correlated state and observation noises, applying state decomposition and backward separation principle (Wang et al. \cite{WWX13,WWX15,WWX18}), rather than Girsanov's measure transformation, are worthy to research. The general solvability of the Riccati equations \eqref{R11} and \eqref{R22} requires systematic study. We will consider these topics in the future research.

\end{document}